\documentclass[a4paper, 11pt,reqno]{amsart}
\usepackage[utf8]{inputenc}
\usepackage{amsfonts,extarrows,comment}
\usepackage{bm}
\usepackage{amsmath,amsthm,amssymb,amsfonts,mathrsfs,upgreek}
\usepackage{stmaryrd}
\usepackage{tikz}
\usetikzlibrary{intersections,calc}
\usetikzlibrary{quotes,angles} 
\usepackage{interval}

\usepackage{dsfont}
\usepackage{mathtools}
\usepackage[all]{xy}
\usepackage{dirtytalk}

\usepackage{graphicx}
\newcommand{\vsimeq}{\mathrel{\rotatebox[origin=c]{-90}{$\simeq$}}}

\usepackage{tikz-cd}
\usepackage{todonotes}
\usepackage{mathtools}
\usepackage{xcolor}

\usepackage{subcaption}
\usepackage[margin=0.8in]{geometry}

\usepackage[justification=centering]{caption}

\usepackage{extarrows}

\usepackage{url}

\definecolor{color/link}{HTML}{4169e1}
\definecolor{color/out}{HTML}{C3CDD6}

\usepackage[colorlinks=true,citecolor=color/link,linkcolor=color/link,urlcolor=color/out,backref=page]{hyperref}

\newif\ifHideFoot
\HideFoottrue  
\HideFootfalse  

\ifHideFoot

\newcommand{\Yiran}[1]{}

\else 

\newcommand{\marg}[1]{\normalsize{{
			\color{red}\footnote{{\color{blue}#1}}}{\marginpar[\vskip
			-.25cm{\color{red}\hfill$\Rightarrow$\tiny\thefootnote}]{\vskip
				-.2cm{\color{red}$\Leftarrow$\tiny\thefootnote}}}}}

\newcommand{\Yiran}[1]{\marg{(Yiran) #1}}
\fi

\usepackage{listings}
\lstdefinestyle{Mathematica}{
    language        =   Mathematica, 
    basicstyle      =   \ttfamily,
    numberstyle     =   \ttfamily,
    keywordstyle    =   \color{blue},
    keywordstyle    =   [2] \color{teal},
    stringstyle     =   \color{magenta},
    commentstyle    =   \color{red}\ttfamily,
    breaklines      =   true,   
    columns         =   fixed,  
    basewidth       =   0.5em,
}

\numberwithin{equation}{section}

1
\numberwithin{subsection}{section}

\allowdisplaybreaks[1]
\usepackage{enumitem}
\setlist[enumerate]{label=\rm{(\roman*)},leftmargin=\parindent,itemindent=\parindent,labelsep=5pt}
\newlist{TFAE}{enumerate}{1}
\setlist[TFAE]{label=\rm{(\alph*)},labelindent=2\parindent}

\newlist{enumeratea}{enumerate}{2}
\setlist[enumeratea]{label=\rm{(\emph{\alph*})}}

\newlist{enumerate1}{enumerate}{2}
\setlist[enumerate1]{label=(\emph{\arabic*})}

\newtheorem*{namedtheorem}{\theoremname}
\newcommand{\theoremname}{testing}

\newtheorem{theorem}{Theorem}[section]
\newtheorem{proposition}[theorem]{Proposition}

\newtheorem{corollary}[theorem]{Corollary}
\newtheorem{lemma}[theorem]{Lemma}

\newtheorem{IH}{\indent Induction Hypothesis}

\theoremstyle{definition}
\newtheorem{definition}[theorem]{Definition}
\newtheorem{definition/proposition}[theorem]{Definition/Proposition}

\newtheorem{notation}[theorem]{Notation}
\newtheorem{example}[theorem]{Example}
\newtheorem{remark}[theorem]{Remark}
\newtheorem*{remark*}{Remark}

\theoremstyle{remark}
\newtheorem*{claim}{Claim}

\newtheoremstyle{italicstyle}   
  {\topsep}                     
  {\topsep}                     
  {\itshape}                    
  {}                            
  {\itshape}                    
  {.}                           
  { }                           
  {\thmname{#1}\thmnumber{ #2}} 

\theoremstyle{italicstyle}

\theoremstyle{plain}
\makeatletter
\newtheorem*{rep@theorem}{\rep@title}
\newcommand{\newreptheorem}[2]{%
	\newenvironment{rep#1}[1]{%
		\def\rep@title{#2 \ref{##1}}%
		\begin{rep@theorem}}%
		{\end{rep@theorem}}}
\makeatother
\newreptheorem{theorem}{Theorem}
\newreptheorem{definition}{Definition}
\newreptheorem{corollary}{Corollary}
\newreptheorem{proposition}{Proposition}
\newreptheorem{question}{Question}

\newcommand\cA{\mathcal{A}}

\newcommand\cC{\mathcal{C}}
\newcommand\cD{\mathcal{D}}

\newcommand\cF{\mathcal{F}}
\newcommand\cG{\mathcal{G}}
\newcommand\cH{\mathcal{H}}

\newcommand\cL{\mathcal{L}}

\newcommand\cN{\mathcal{N}}
\newcommand\cO{\mathcal{O}}
\newcommand\cP{\mathcal{P}}

\newcommand\cT{\mathcal{T}}

\newcommand\cX{\mathcal{X}}

\newcommand\cZ{\mathcal{Z}}

\newcommand\CC{\mathbb{C}}

\newcommand\PP{\mathbb{P}}
\newcommand\QQ{\mathbb{Q}}
\newcommand\RR{\mathbb{R}}

\newcommand\ZZ{\mathbb{Z}}

\newcommand\rR{\mathrm{R}}

\newcommand\rmb{\mathrm{b}}

\newcommand\fS{\mathfrak{S}}

\newcommand\va{a(E)}
\newcommand\vb{b(E)}
\newcommand\vc{c(E)}
\newcommand\vd{d(E)}
\newcommand\I{\sqrt{-1}}

\DeclareMathOperator{\alb}{alb}
\DeclareMathOperator{\Alb}{Alb}

\DeclareMathOperator{\CH}{CH}

\DeclareMathOperator{\D}{D}
\DeclareMathOperator{\id}{\mathds{1}}
\DeclareMathOperator{\Pic}{Pic}

\DeclareMathOperator{\Hom}{Hom}

\DeclareMathOperator{\SL}{SL}

\DeclareMathOperator{\spec}{Spec}

\DeclareMathOperator{\ch}{ch}
\DeclareMathOperator{\td}{td}

\DeclareMathOperator{\Aut}{Aut}

\DeclareMathOperator{\Stab}{Stab}

\DeclareMathOperator{\Coh}{Coh}

\DeclareMathOperator{\dCat}{D}

\DeclareMathOperator{\Hilb}{Hilb}
\DeclareMathOperator{\GL}{GL}

\DeclareMathOperator{\rk}{rk}
\DeclareMathOperator{\coho}{\rm H}

\DeclareMathOperator{\cone}{\mathrm{Cone}}

\DeclareMathOperator{\cHom}{\hspace{-0.1em}\mathcal{H}\hspace{-0.05em}\mathit{om}}
\DeclareMathOperator{\RcHom}{\rR\cHom}

\DeclareMathOperator{\supp}{supp}
\DeclareMathOperator{\Supp}{Supp}

\DeclareMathOperator{\K}{K}

\newcommand{\can}{\mathrm{can}}
\newcommand{\sub}{\mathrm{sub}}
\newcommand{\quot}{\mathrm{quot}}
\newcommand{\nat}{\mathrm{nat}}
\newcommand{\num}{\mathrm{num}}

\newcommand{\act}{\mathrm{act}}

\newcommand{\diag}{\mathrm{diag}}

\newcommand{\defeq}{\vcentcolon=}
\newcommand{\eqdef}{=\vcentcolon}

\makeatletter
\newcommand*{\da@rightarrow}{\mathchar"0\hexnumber@\symAMSa 4B }
\newcommand*{\da@leftarrow}{\mathchar"0\hexnumber@\symAMSa 4C }
\newcommand*{\xdashrightarrow}[2][]{%
  \mathrel{%
    \mathpalette{\da@xarrow{#1}{#2}{}\da@rightarrow{\,}{}}{}%
  }%
}
\newcommand{\xdashleftarrow}[2][]{%
  \mathrel{%
    \mathpalette{\da@xarrow{#1}{#2}\da@leftarrow{}{}{\,}}{}%
  }%
}
\newcommand*{\da@xarrow}[7]{%
  \sbox0{$\ifx#7\scriptstyle\scriptscriptstyle\else\scriptstyle\fi#5#1#6\m@th$}%
  \sbox2{$\ifx#7\scriptstyle\scriptscriptstyle\else\scriptstyle\fi#5#2#6\m@th$}%
  \sbox4{$#7\dabar@\m@th$}%
  \dimen@=\wd0 %
  \ifdim\wd2 >\dimen@
    \dimen@=\wd2 %
  \fi
  \count@=2 %
  \def\da@bars{\dabar@\dabar@}%
  \@whiledim\count@\wd4<\dimen@\do{%
    \advance\count@\@ne
    \expandafter\def\expandafter\da@bars\expandafter{%
      \da@bars
      \dabar@ 
    }%
  }%
  \mathrel{#3}%
  \mathrel{%
    \mathop{\da@bars}\limits
    \ifx\\#1\\%
    \else
      _{\copy0}%
    \fi
    \ifx\\#2\\%
    \else
      ^{\copy2}%
    \fi
  }%
  \mathrel{#4}%
}
\makeatother

\newcommand\xdhrightarrow[2][]{%
  \mathrel{\ooalign{$\xrightarrow[#1\mkern4mu]{#2\mkern4mu}$\cr%
  \hidewidth$\rightarrow\mkern4mu$}}
}

\title[Stability conditions on some higher-dimensional varieties]{Bridgeland stability conditions on some higher-dimensional \\ Calabi--Yau manifolds and generalized Kummer varieties}
\author{Yiran Cheng}
\begin{document}
\begin{abstract}
We construct Bridgeland stability conditions on the the following hyper-K\"ahler or strict Calabi--Yau manifolds:\\
- Generalized Kummer varieties associated to an abelian surface that is isogenous to a product of elliptic curves.\\
- Universal covers of Hilbert schemes of some Enriques surfaces; this provides examples of stability conditions on strict Calabi--Yau manifolds in each even dimension.\\
- Albanese fibers of some finite \'etale covering of Hilbert schemes of some bielliptic surfaces; this provides examples of stability conditions on strict Calabi-Yau manifolds in each odd dimension.\\
- Cynk--Hulek Calabi--Yau manifolds with an automorphism of order $2$ or $3$.

\end{abstract}
\maketitle

\setcounter{tocdepth}{1} 
\tableofcontents

\section{Introduction}
The concept of stability conditions on $\CC$-linear triangulated categories was introduced by Bridgeland in his seminal work \cite{Bri07}, establishing a rigorous mathematical framework inspired by Douglas's pioneering research on D-branes and $\Pi$-stability \cite{Douglas:2002}.
Heuristically, a stability condition may be viewed as a noncommutative analogue of an ample line bundle in algebraic geometry, in that it allows for the formulation of moduli spaces of semistable objects.
This theory has led to many applications, including a description of autoequivalence group of K3 surfaces \cite{Bri08,Bayer-Bridgeland:2017}, answers to Brill--Noether questions \cite{Bay18,Feyzbakhsh:2020,Feyzbakhsh:2024,Cheng-Li-Wu:2025},
structure theorems for Donaldson--Thomas invariants \cite{Toda:2010,Feyzbakhsh-Thomas:2023a,Feyzbakhsh-Thomas:2023,Feyzbakhsh-Thomas:2024}, and advances in our understanding of hyper-K\"{a}hler geometry \cite{Bayer-Macri:2014a,Bayer-Macri:2014,BLM+21,PPZ22}, etc.

However, the construction of stability conditions is challenging in general: although significant progress has been made in the cases of curves and surfaces \cite{Bri08,Arcara-Bertram:2013}, the existence of stability conditions on higher-dimensional varieties is often considered as one of the biggest open problems in the study of derived categories in recent years.
Traditionally, the construction of stability conditions relies on proving certain generalized Bogomolov--Gieseker type inequality \cite{Bayer-Macri-Toda:2014,Bernardara-Macri-Schmidt+:2017}, which is still mysterious in general.
Even in the $3$-dimensional case, this is only known for some special varieties, like Fano threefolds \cite{Li:2019,Bernardara-Macri-Schmidt+:2017}, abelian threefolds and some quotients \cite{Maciocia-Piyaratne:2015,Maciocia-Piyaratne:2016,BMS16}, and quintic threefolds \cite{Li:2019a}.
On the other hand, one can also induce new stability conditions on varieties from existing ones. In this direction, there is some study on finite quotients \cite{Pol07,MMS09,PPZ23,Del23} and product with a curve \cite{Liu21}.

This paper follows the second path and investigates more techniques to induce stability conditions. Namely, we study inducing stability conditions via group quotient, finite cover, product, and more importantly, via restriction.
Combining these techniques, we construct stability conditions on some higher-dimensional varieties from product of curves.

\subsection{Main applications}

As applications, we construct stability conditions on the derived category of the following higher-dimensional varieties which are not known to carry stability conditions before.

\subsubsection*{Generalized Kummer varieties \cite{Beauville:1983}}
In the forthcoming paper \cite{LMPSZ}, the authors constructed stability conditions on some hyper-K\"{a}hler manifolds of $\mathrm{K3}^{[n]}$-type. Based on \cite{LMPSZ}, we further develop the techniques by studying restriction problems for stability condtions, and construct stability conditions on the following family of hyper-K\"{a}hler manifolds of generalized Kummer type.
\begin{theorem}[\textbf{Generalized Kummer varieties}]\label{thm:Kummer}
    Let $S$ be an abelian surface which is isogenous to a product of elliptic curves, i.e., $S \simeq (C_1 \times C_2)/G$, where $C_i$ ($i=1,2$) are elliptic curves and $G$ is a finite group acting freely on $C_1 \times C_2$. 
    Consider the generalized Kummer variety $K_n(S)$ associated to the abelian surface $S$, i.e.,
    \begin{equation*}
    K_n(S) = \alpha_n^{-1}(0),
    \end{equation*}
    where $\alpha_n$ is the Albanese map for the Hilbert scheme $S^{[n]}$. Then, there exist stability conditions on the derived category of $K_n(S)$.
\end{theorem}
In the work
\cite{LMPSZ}, the authors show that the stability conditions can always deform to very general fibers in a family.
Combining this with Theorem~\ref{thm:Kummer}, we have the following consequence.
\begin{corollary}
    There exist stability conditions on the derived category of a very general projective hyper-K\"{a}hler manifold of generalized Kummer type.
\end{corollary}

\subsubsection*{Calabi--Yau manifolds via Hilbert schemes of points \cite{OS11}}

The following two theorems together provide examples of strict Calabi--Yau manifolds in arbitrary dimension whose derived categories admit stability conditions.

\begin{theorem}[\textbf{Even dimensional Calabi--Yau manifolds arising from Enriques surfaces}]\label{thm:CY even}
    Let $C_i~(i=1,2)$ be two elliptic curves.
    Consider the actions
    \begin{equation*}
        \begin{aligned}
                -\id \colon C_1 \times C_2 &\to C_1 \times C_2 \\
                (x,y) &\mapsto (-x,-y) \\
                \text{and} \quad
                \tau \colon C_1 \times C_2 &\to C_1 \times C_2 \\
                (x,y) &\mapsto (-x+t_1,y+t_2)
        \end{aligned}
    \end{equation*}
    where $t_i$ is any fixed 2-torsion point on $C_i$.
    We denote the resolution of $(C_1 \times C_1) / \langle -\id \rangle$ by $S'$, which is well-known as a Kummer (K3) surface,
    and the resolution of $(C_1 \times C_2) / \langle -\id ,\tau \rangle$ by $S$.
    One easily checks that $\tau$ descends to $(C_1 \times C_1) / \langle -\id \rangle$ and lifts to a free involution on $S'$ which we will still denote by $\tau$. Therefore, $S = S' / \langle \tau \rangle$ is an Enriques surface.
    Now consider the Hilbert scheme of $n$ points for $n \geq 2$. According to \cite[Theorem 3.1]{OS11}, the universal cover $f \colon \widetilde{S^{[n]}} \to {S^{[n]}}$ is $2:1$ and $\widetilde{S^{[n]}}$ is a strict Calabi--Yau manifold.
    Then, there exist stability conditions on the derived category of $\widetilde{S^{[n]}}$.
\end{theorem}


\begin{theorem}[\textbf{Odd dimensional Calabi--Yau manifolds arising from bielliptic surfaces}] \label{thm:CY odd}
    Let $S$ be a bielliptic surface, i.e., $S \simeq (C_1 \times C_2)/G$ where $C_i~(i=1,2)$ are elliptic curves and $G$ is a finite group acting by translation on $C_1$ and by automorphism (of elliptic curves) on $C_2$. 
    Now consider the Hilbert scheme of $n$ points for $n \geq 2$. 
    According to \cite[Theorem~3.5]{OS11},
    we can find a finite \'{e}tale covering $f \colon \widetilde{S^{[n]}} \to S^{[n]}$ such that $\widetilde{S^{[n]}}\simeq X \times C_2$ and $X$ is a strict Calabi--Yau manifold.
    Then, there exist stability conditions on the derived category of $X$.
\end{theorem}

\subsubsection*{Cynk--Hulek Calabi--Yau manifolds \cite{Cynk-Hulek:2007}}
Let $C_1, \dots , C_n$ be elliptic curves (respectively, elliptic curves with order $3$ automorphisms), and 
\begin{equation*}
    G_n \defeq \{ (a_1, \dots, a_n) \mid \sum a_i =0 \} \subset (\ZZ/2)^n \qquad \text{(respectively, $\subset (\ZZ/3)^n$)}
\end{equation*}
act on $C_1 \times \dots \times C_n$ via factor-wise involution (respectively, via the order $3$ automorphism).
A Cynk--Hulek Calabi--Yau manifold arising from $\ZZ/2$-actions (respectively, $\ZZ/3$-actions) is constructed as a crepant resolution of $(C_1 \times \dots \times C_n) /{G_n}$; see Section~\ref{section:CH} (respectively, Section~\ref{section:CH Z/3}) for more details.

The following theorem provides another two families of strict Calabi--Yau manifolds in arbitrary dimension whose derived categories admit stability conditions.
\begin{theorem}[\textbf{Cynk--Hulek Calabi--Yau manifolds}]\label{thm:CH}    
    Let $X_n$ be the $n$-dimensional Cynk--Hulek Calabi--Yau manifold arising from $\ZZ/2$-actions or $\ZZ/3$-actions on elliptic curves. Then there exist stability conditions on the derived category of $X_n$.
\end{theorem}

\begin{remark}
    The case of $\ZZ/2$-action was previously obtained by \cite{PS24} and our argument is similar to theirs. The case of $\ZZ/3$-action is technically more involved.
\end{remark}



\subsection{Restricting stability conditions}
The main novelty of the paper is a technique allowing us to \emph{restrict} stability conditions to closed subvarieties in certain cases. 
Recall that a stability condition $\sigma \in \Stab(X)$ is called geometric if for every point $x\in X$ the skyscraper sheaf $\cO_x$ is $\sigma$-stable, and all skyscraper sheaves are of the same phase.

\begin{theorem} [Theorem~\ref{thm:main}]
    Let $f \colon X \to Y$ be a morphism between smooth projective varieties such that the Albanese morphism of $Y$ is finite. 
    Let $X_0$ be a fiber of $f$ such that $f$ is isotrivial near the fiber $X_0$, and let $i\colon X_0 \to X$ be the closed immersion.
    Then, there is a natural restriction 
    \begin{equation*}
        \Stab_{\Lambda}(X) \to \Stab_{\Lambda_0}(X_0), \qquad (\cP ,Z) \mapsto (\cP_0, Z_0),
    \end{equation*}
    where
    \begin{equation*}
        \cP_0(\phi) = (i_*)^{-1} \cP(\phi) 
        \quad \text{and} \quad  Z_0 =Z \circ [i_*] .
    \end{equation*}
    Here, if $(\cP,Z)$ satisfies the support property with respect to $\K(X) \xdhrightarrow{v} \Lambda$,    
    then $(\cP_0,Z_0)$ satisfies the support property with respect to
    $\K( X_0 ) \xdhrightarrow{} \Lambda_0$, which is defined as the coimage component of the composition $\K(X_0) \xrightarrow{[i_*]} \K(X) \xdhrightarrow{v} \Lambda$.
    In particular, the restriction respects geometric stability conditions.
\end{theorem}

The proof of this theorem proceeds in the following two steps.
Throughout this paper, we use the notation $\dCat(X)$ for the derived category of $X$, and $\dCat_Y(X)$ for its full subcategory whose objects are (set-theoretically) supported on a closed subset $Y$ of $X$; see Notation~\ref{notation} for more details.

\textbf{Step 1} As observed in \cite{LMPSZ}, for any stability condition $\sigma$ on $\dCat(X)$ and any object $F$ in $\dCat_{X_0}(X)$, all the Harder–Narasimha factors of $F$ with respect to $\sigma$ lie in $\dCat_{X_0}(X)$ as well.
In other words, we can restrict a stability condition from $\dCat(X)$ to $\dCat_{X_0}(X)$.
We summarize the related results from \cite{LMPSZ} in Section~\ref{section:restrict set} without proof. 

\textbf{Step 2} (Proposition~\ref{restriction})
Furthermore, if the object $F$ lies in the essential image of $\dCat(X_0)$, then the Harder--Narasimhan (HN) filtration of $F$ with respect to $\sigma$ also comes from $\dCat(X_0)$.
In other words, we can actually restrict the stability condition $\sigma$ further to $\dCat(X_0)$.

\bigskip

We also get the following by-product which addresses the stability of pushforward objects.
\begin{corollary}[Corollary~\ref{pushforward of stable obj g>1}]
    Let $F \in \dCat(X)$ be $\sigma$-stable, where $\sigma \in \Stab(X)$, and let $i \colon X \to X \times C$ be the inclusion map of any fiber of the natural projection $q \colon X \times C \to C$, where $C$ is any smooth projective curve of genus at least one. 
    Then for any $\sigma^{s,t,\beta} \in \Stab(X \times C)$ constructed by Liu \cite{Liu21} (see Theorem~\ref{stabilities on product}), $i_*F$ is $\sigma^{s,t,\beta}$-stable.
\end{corollary}

\subsection{Organization of the paper}
In Section~\ref{sec:pre} we review basics on stability conditions and results on inducing stability conditions via group actions.
In Section~\ref{sec:product} we review the construction of constant $t$-structures on product spaces by Abramovich and Polishchuk \cite{AP06}, as well as the construction of stability conditions on product with a curve by Liu \cite{Liu21}.
In Section~\ref{sec:restrict} we develop techniques to restrict a stability condition to a good closed subvariety.
In Section~\ref{sec:example} we combine all the techniques and give many new examples of stability conditions on higher-dimensional varieties.

\subsection{Acknowledgements} 
I would like to thank Lie Fu for suggesting this interesting project and for his continuous guidance.
I am profoundly grateful to Chunyi Li and Xiaolei Zhao for many insightful discussions and for generously sharing unpublished ideas.
I would also like to thank Soheyla Feyzbakhsh, Yucheng Liu, Emanuele Macr\`{i}, Alexander Perry, Paolo Stellari, Richard Thomas, and Ruxuan Zhang for helpful discussions.

\section{Preliminaries}\label{sec:pre}

\subsection{$t$-structures and stability conditions}
We briefly review the definition of $t$-structures and stability conditions \cite{Bri07,KS08}. 
Let $\cD$ be a triangulated category.
\begin{definition}
    A \emph{$t$-structure} on $\cD$ is a pair of full subcategories $(\cD^{\leq 0},\cD^{\geq 0})$ satisfying the following conditions:
    \begin{enumerate}
        \item $\cD^{\leq 0}[1] \subset \cD^{\leq 0}$ and $\cD^{\geq 0}[-1] \subset \cD^{\geq 0}$;
        \item $\Hom(T,F)=0$, for every $T\in \cD^{\leq 0}$ and $F \in \cD^{\geq 1}$;
        \item For any object $E\in \cD$, there exists a distinguished triangle
        \begin{equation}\label{eq:t}
            T_E \to E \to F_E \xdashrightarrow{+1}
        \end{equation}
        with $T_E\in \cD^{\leq 0}$ and $F_E \in \cD^{\geq 1}$.
    \end{enumerate}
\end{definition}
Here we used the notation $\cD^{\leq n} \defeq \cD^{\leq 0}[-n]$ and $\cD^{\geq n}\defeq \cD^{\geq 0}[-n]$.
We will also write $\cD^{[a,b]}=\cD^{\leq b} \cap \cD^{\geq a}$, for all $a,b \in \ZZ \cup \{\pm \infty \}$ with $a \leq b$.

\begin{remark}
    This definition is a direct analogue of the notion of a \emph{torsion pair} $(\mathcal{T}, \mathcal{F})$ 
    for an abelian category: condition (ii) corresponds to $\Hom(\mathcal{T}, \mathcal{F}) = 0$, 
    and condition (iii) corresponds to the existence for every object $E$ of a short exact sequence 
    $0 \to T_E \to E \to F_E \to 0$ with $T_E \in \mathcal{T}$ and $F_E \in \mathcal{F}$.
    The additional condition (i) is specific to the triangulated setting.
\end{remark}

The inclusion $\cD^{\leq n} \subset \cD$ has a right adjoint \emph{truncation functor} $\tau^{\leq n} \colon \cD \to \cD^{\leq n}, E \mapsto \tau^{\leq n} E \defeq T_E $ (as in \eqref{eq:t}). 
Similarly, $\tau^{\geq n} \colon \cD \to \cD^{\geq n}$ defines a left adjoint to the inclusion $\cD^{\geq n} \subset \cD$. Thus, \eqref{eq:t} can be written as a distinguished triangle
\begin{equation*}
    \tau^{\leq 0} E \to E \to \tau^{\geq 1} E \xdashrightarrow{+1}.
\end{equation*}

\begin{definition}
    The \emph{heart} of a $t$-structure $(\cD^{\leq 0},\cD^{\geq 0})$ is defined as 
    \begin{equation*}
        \cA \defeq \cD^{\leq 0} \cap \cD^{\geq 0} = \cD^{[0,0]}.
    \end{equation*}
\end{definition}
The heart $\cA$ is an abelian category and short exact sequences in $\cA$ are precisely the distinguished triangles in $\cD$ with all terms in $\cA$, cf. \cite{Beilinson-Bernstein-Deligne:1982}. 
Moreover, one has \emph{cohomology functors}
\begin{equation*}
    \cH^i  \colon \cD \to \cA , \qquad E \mapsto (\tau^{\geq i} \tau^{\leq i}E)[i].
\end{equation*}

\begin{definition}
    Let $\cC$ and $\cD$ be a pair of triangulated categories equipped with $t$-structures.
    A triangulated functor $\Phi \colon \cC \to \cD$ is called \emph{left} (resp. \emph{right}) \emph{$t$-exact} if $\Phi(\cC^{\geq 0}) \subset \cD^{\geq 0}$ (resp. $\Phi(\cC^{\leq 0}) \subset \cD^{\leq 0}$). A \emph{$t$-exact functor} is a functor which is both left and right $t$-exact. 
\end{definition}

\begin{remark}\label{conservatively t-exact}
    Recall that a triangulated functor $\Phi \colon \cC \to \cD$ between triangulated categories is called \emph{conservative}\footnote{The terminology follows \cite{BLM+21}.} if $\Phi(E) \simeq 0$ implies $E \simeq 0$. If $\Phi$ is both conservative and $t$-exact, then $\Phi(E) \in \cD^{[a,b]}$ if and only if $E\in \cC^{[a,b]}$.
\end{remark}

\begin{definition}
    A $t$-structure $(\cD^{\leq 0},\cD^{\geq 0})$ is called \emph{bounded}\footnote{Our terminology follows \cite{Bri07,BLM+21}. It is called nondegenerate in \cite{AP06,Pol07}, and bounded and nondegenerate in \cite{Beilinson-Bernstein-Deligne:1982}.} if $\cD=\bigcup_{n,m\in \ZZ}\cD^{\leq n}\cap \cD^{\geq m}$.
\end{definition}
\begin{remark}[{\cite{LC07}, \cite[Remark 1.15]{Huy14}}]\label{Krull--Schmidt}
    Not every triangulated category $\cD$ admits a bounded $t$-structure. In fact, the existence of a bounded t-structure on $\cD$ implies that $\cD$ is idempotent closed (or, Karoubian). 
    As a consequence, let $R$ be a commutative noetherian ring which is complete and local, and let $\cA$ be an Ext-finite abelian category over $R$. Then the bounded derived category $\dCat^\rmb (\cA)$ is a Krull--Schmidt category\footnote{Recall that an additive category is called Krull-Schmidt category if every object decomposes into a finite direct sum of objects having local endomorphism rings. This implies that up to permutation, any object has a unique way to write as a direct sum of indecomposable objects.}. In particular, for any smooth projective variety $X$ (over a field), its derived category $\dCat(X)$ is Krull--Schmidt. 
\end{remark}

A bounded $t$-structure is uniquely determined by its heart.
\begin{lemma}[{\cite[Lemma 3.1]{Bri08}}]
    If $\cA \subset \cD$ is a full additive subcategory of a triangulated category $\cD$, then $\cA$ is the heart of a bounded $t$-structure on $\cD$ if and only if
    \begin{enumerate}
        \item For any $k_1>k_2$ and $E_i \in \cA[k_i]$, we have $\Hom_\cD(E_1,E_2)=0$;
        \item Every nonzero object $E \in \cD$ fits into a diagram of a finite collection of distinguished triangles
        \begin{equation*}
            \begin{tikzcd}[column sep=tiny]
            0=E_0 \arrow[rr] &                        & E_1 \arrow[rr] \arrow[ld] &                        & E_2 \arrow[rr] \arrow[ld] &  & \cdots \arrow[rr] &  & E_{n-1} \arrow[rr] &                        & E_n=E \arrow[ld] \\
                             & A_1 \arrow[lu, dashed, "+1"] &                           & A_2 \arrow[lu, dashed, "+1"] &                           &  &                   &  &                    & A_n \arrow[lu, dashed, "+1"] &                 
            \end{tikzcd}
        \end{equation*}
        with  $A_i \in \cA[k_i]$ for some sequence $k_1>k_2>\dots >k_n$ of integers.
    \end{enumerate}
\end{lemma}
This formulation sheds light on the following definition.

\begin{definition}
    A \emph{slicing} $\cP$ on a triangulated category $\cD$ consists of full additive subcategories $\cP(\phi) \subset \cD$ for each $\phi \in \RR$, satisfying the following conditions:
    \begin{enumerate}
        \item For any $\phi\in \RR$, $\cP(\phi+1)=\cP(\phi)[1]$;
        \item If $\phi_1> \phi_2$ and $E_i \in \cP(\phi_i)$, then $\Hom_\cD(E_1,E_2)=0$;
        \item Every nonzero object $E\in \cD$ has a Harder--Narasimhan (HN) filtration, i.e., there exists a finite collection of distinguished triangles
        \begin{equation*}
            \begin{tikzcd}[column sep=tiny]
            0=E_0 \arrow[rr] &                        & E_1 \arrow[rr] \arrow[ld] &                        & E_2 \arrow[rr] \arrow[ld] &  & \cdots \arrow[rr] &  & E_{n-1} \arrow[rr] &                        & E_n=E \arrow[ld] \\
                             & A_1 \arrow[lu, dashed, "+1"] &                           & A_2 \arrow[lu, dashed, "+1"] &                           &  &                   &  &                    & A_n \arrow[lu, dashed, "+1"] &                 
            \end{tikzcd}
        \end{equation*}
        with  $A_i \in \cP(\phi_i)$ for some sequence $\phi_1>\phi_2>\dots >\phi_n$ of real numbers.
    \end{enumerate}
\end{definition}
We write $\phi^+(E) \defeq \phi_1$ and $\phi^-(E) \defeq \phi_m$; moreover, for an interval $I \subset \RR$, we define the following full subcategory
\begin{equation*}
    \cP(I)\defeq \{ E \mid \phi^+(E),\phi^-(E) \in I \}= \langle \cP(\phi) \rangle_{\phi \in I} \subset \cD.
\end{equation*}
We also write $\cP(\leq a) \defeq \cP((\infty,a])$ or $\cP(>b)\defeq \cP((b,\infty))$.
\begin{definition}
    A \emph{pre-stability condition} on a triangulated category $\cD$ is a pair $\sigma=(\cP,Z)$, where $\cP$ is a slicing of $\cD$ and $Z \colon \K(\cD)\to \CC$ is a group homomorphism from the Grothendieck group of $\cD$ to complex numbers, that satisfy the following condition:
        \begin{center}
            for any nonzero $E \in \cP(\phi)$, we have $Z([E]) \in \RR_{>0} \cdot e^{i\pi \phi}$.
        \end{center}
    The nonzero objects of $\cP(\phi)$ are called \emph{$\sigma$-semistable} of phase $\phi$, among which the simple objects are called \emph{$\sigma$-stable} of phase $\phi$. We call $Z$ the \emph{central charge} and will often abuse notation to write $Z(E)$ for $Z([E])$. 
\end{definition}

\begin{definition}
    A pre-stability condition $\sigma=(\cP,Z)$ on $\cD$ satisfies the \emph{support property} (with respect to a surjective group homomorphism $v:\K(\cD) \xdhrightarrow{} \Lambda$ where $\Lambda$ is a finite rank abelian group) if there exists a constant $C>0$ such that for all $\sigma$-semistable objects $0 \neq E \in \cD$ we have
    \begin{equation*}
        \|v(E)\| \leq C |Z(E)|
    \end{equation*}
    where $\| - \|$ is any norm on $\Lambda \otimes \RR$.
    A pre-stability condition that satisfies the support property is called a \emph{stability condition}. If moreover $v$ factors through the numerical Grothendieck group $\K_{\num}(\cD) \defeq \K(\cD) / \ker \chi$, we call it a \emph{numerical stability condition}.
\end{definition}

The set of stability conditions with respect to $v\colon \K(\cD) \to \Lambda$ will be denoted by $\Stab_\Lambda(\cD)$ if $v$ is clear from context. Recall that by an observation of Bridgeland, this set is naturally equipped with the structure of a complex manifold, such that the forgetful map
\begin{align*}
    \cZ \colon \Stab_\Lambda(\cD) &\to \Hom(\Lambda,\CC) , \\
    \sigma=(\cP,Z) &\mapsto Z
\end{align*}
is a local biholomorphism, cf. \cite[Theorem 7.1]{Bri07} and \cite[Proposition A.5]{BMS16}. 

\begin{remark}\label{projection to slicings}
    There is an equivalent way to define a (pre-)stability condition as a pair $\sigma=(\cA,Z)$, where $\cA\subset \cD$ is the heart of a bounded $t$-structure, satisfying analogous conditions, cf. \cite[Section~2]{Bri07}. It is easy to check the equivalence by the assignment $\cA=\cP((0,1])$. One advantage of using slicing is that it is easier to see the assignment (basically the projection to the other factor)
    \begin{align*}
        \Stab_\Lambda(\cD) &\to \left\{ \text{(bounded) $t$-structures $\left((\cD_\phi^{\leq 0},\cD_\phi^{\geq 0})\right)_{\phi \in \RR}$ parametrized by $\phi \in \RR$} \right\} , \\
        \sigma=(\cP,Z) &\mapsto \Big( \big(\cP(>\phi), \cP(\leq \phi+1) \big) \Big)_{\phi \in \RR}
    \end{align*}
    determines $\sigma$ up to a scalar in $\RR_{>0}$.
\end{remark}

In the geometric setting, we adopt the following notation.

\begin{notation} \label{notation}
    For any noetherian scheme $X$, we write $\dCat(X) \defeq \dCat^\rmb(\Coh(X))$ for the bounded derived category of coherent sheaves on $X$.
    It is naturally equipped with the structure of a triangulated category \cite[IV.2]{Gelfand-Manin:1996}
    and we will write $\Stab(X)\defeq \Stab(\dCat(X))$.
    Moreover, if $Y$ is a closed subset of $X$, we write $\dCat_Y(X) \subset \dCat(X)$ consisting of all complexes whose cohomology are supported on $Y$,
    or equivalently,
    the bounded derived category $\dCat^\rmb(\Coh_Y(X))$ of coherent sheaves supported on $Y$,
    \cite[Lemma~2.1]{Orl11}.
    We will write $\Stab(Y)_X \defeq \Stab \left( \dCat_Y(X) \right)$.
\end{notation}

In this setting we have the following definition.
\begin{definition}
    A stability condition $\sigma \in \Stab(X)$ is called \emph{geometric} if for every point $x\in X$, the skyscraper sheaf $\cO_x$ is $\sigma$-stable, and all skyscraper sheaves are of the same phase.
\end{definition}
\begin{remark}\label{geometric}
    By \cite[Proposition~2.9]{FLZ22}, for a numerical stability condition $\sigma$, if all skyscraper sheaves are $\sigma$-stable, then they must have the same phase; in other words, $\sigma$ is geometric.
\end{remark}

\subsection{Equivariant categories}
We briefly review the basics on group action on categories and the associated equivariant categories.
Throughout we assume $G$ is a finite group.
Unlike a group action on sets, if we consider a $G$-action on some category $\cC$, it is not only the data of $G \to \Aut(\cC)$, because $\Aut(\cC)$ contains more data of a group -- there are also 2-morphisms between autoequivalences. Following \cite{Beckmann-Oberdieck:2023}, we have following definition.
\begin{definition}
    Let $G$ be a finite group and $\cC$ be a category.
    An action of $G$ on $\cC$ consist of the following data
    \begin{itemize}
        \item For every $g\in G$, a functor ($1$-morphism) $\phi_g \colon \cC \to \cC$, such that $\phi_1=\id_\cC$.
        \item For every pair $f,g\in G$, a natural equivalence ($2$-morphism) $\phi_{g,f}\colon \phi_g \circ \phi_f  \Rightarrow \phi_{gf}$, such that the two natural ways to compose $\phi_h \circ \phi_g \circ \phi_f \Rightarrow \phi_{hgf}$ coincide.
    \end{itemize}
\end{definition}

Once we have a group action, we can define the equivariant category.

\begin{definition}
Suppose $G$ is a finite group acting on an additive $\CC$-linear category $\cC$. The \emph{equivariant category} $\cC^G$ is defined as follows:
\begin{itemize}
    \item Objects are pairs $(E,\lambda)$ where $E$ is an object of $\cC$ and $\lambda$ is a collection of isomorphisms $\lambda_g \colon E \xrightarrow{\sim} \phi_g(E)$ for $g\in G$ such that the diagram
    \begin{equation*}
    \begin{tikzcd}
    E \arrow[r, "\lambda_g"] \arrow[rrr, "\lambda_{gh}", bend right] & \phi_g(E) \arrow[r, "\phi_g\lambda_h"] & \phi_g\phi_h(E) \arrow[r, "{\phi_{g,h}^E}"] & \phi_{gh}(E)
    \end{tikzcd}
    \end{equation*}
    commutes for all $g,h \in G$. We will call $\lambda$ a \emph{linearization}.
    \item Morphisms from $(E, \lambda)$ to $(E',\lambda')$ are morphisms from $E$ to $E'$ which commutes with linearizations, i.e. such that
    \begin{equation*}
        \begin{tikzcd}
        E \arrow[r, "f"] \arrow[d, "\lambda_g"'] & E' \arrow[d, "\lambda_g'"] \\
        \phi_g(E) \arrow[r, "\phi_g(f)"]              & \phi_g (E')                 
        \end{tikzcd}
    \end{equation*}
    commutes for all $g \in G$.
\end{itemize}
\end{definition}

\begin{remark}
    There is a more fancy way to define this in the language of stable $\infty$-categories, see, for example, \cite{BP23,PPZ23}. We will stay in the classical setting in this paper.
\end{remark}

We states the following lemma which says invariants can be computed successively along normal subgroups and their quotients.
\begin{lemma}\label{successive invariants}
        Let
    \begin{equation*}
        1 \to H \to G \to G/H \to 1
    \end{equation*}
    be a short exact sequence of finite groups, and suppose $\cC$ is an additive $\CC$-linear category admitting a $G$-action.
    Then under the natural $H$-action, the equivariant category $\cC^H$ admits a natural $G/H$-action and we have a canonical equivalence
    \begin{equation*}
        \cC^G \simeq (\cC^H)^{G/H}.
    \end{equation*}
\end{lemma}

\begin{proof}
    See, for example, \cite[Proposition~3.3]{Beckmann-Oberdieck:2023}.
\end{proof}

From now on we consider the geometric setting, i.e. take $\cC=\dCat(X)$ and study $\dCat(X)^G$.
If the action is induced by an action of $G$ on $X$, we also have
\begin{equation*}
    \dCat([X/G]) \simeq \dCat(X)^G,
\end{equation*}
where $[X/G]$ is the quotient stack.

\begin{remark}
    In this case, it is naturally to ask whether $\dCat(X)^G \simeq \dCat^\rmb( \Coh(X)^G)$.
    This is true in good situation, for example finite group action in our case. See, for example, \cite[Theorem~3]{Yun:2006}.
\end{remark}
If $G$ acts freely on $X$, then $X/G$ is smooth as well.
In this case, the equivariant category $\dCat([X/G])$ is equivalent to $\dCat(X/G)$.
More precisely, let $\pi \colon X \to X/G$ be the quotient map. One has natural equivalences
\begin{equation}\label{eq:equivalence naive}
    \begin{tikzcd}
        \dCat(X/G) \arrow[rr, "{(\pi^*, \lambda_\nat)}", bend left] &  & \dCat([X/G]) \arrow[ll, "\pi_*(-)^G", bend left] ,
    \end{tikzcd}
\end{equation}
where $(\pi^*, \lambda_\nat)$ is induced by pullback along $X \xrightarrow{\pi} X/G$ together with the natural linearization given by pullback induced by group action of $G$,
and 
$\pi_*(-)^G$ is induced by taking the $G$-invariant part (identifying sections via $\lambda$) of the pushforward.

For non-free actions, the quotient $X/G$ may have singularities. Nevertheless, under certain condition, the following celebrated theorem of Bridgeland--King--Reid suggests that one may still establish an equivalence between $\dCat([X/G])$ and the derived category of $Y$, where $Y$ is a crepant resolution of $X/G$.
\begin{theorem}[{\cite[Theorem~1.1]{BKR01}}]\label{BKR}
    Let $X$ be a quasi-projective variety of dimension $n$ and 
    let $G\subset \Aut(X)$ be a finite group of automorphisms of $X$ such that $\omega_X$ is locally trivial as a $G$-sheaf.
    Let $Y = \Hilb^G_\can (X) \subset \Hilb^G (X)$ be the irreducible component of the Hilbert scheme of $G$-clusters on $X$ containing the free orbits.
    Write $\cZ$ for the universal closed subscheme $Z \subset Y \times X$, and $p$ and $q$ for its projections to Y and X.
    There is a commutative diagram of schemes
    \begin{equation}\label{eq:diagram BKR}
\begin{tikzcd}
                      & \cZ \arrow[ld, "p"'] \arrow[rd, "q"] &                     \\
Y \arrow[rd, "\tau"'] &                                      & X \arrow[ld, "\pi"] \\
                      & X/G                                  &                    
\end{tikzcd}
    \end{equation}
    in which $q$ and $\tau$ are birational, $p$ and $\pi$ are finite, and $p$ is flat.
    We may view $X/G$ and $Y$ are equipped with the trivial $G$-action, hence all the morphisms are $G$-equivariant.
    Define the functor
    \begin{equation}\label{eq:Bondal--Orlov}
        \Phi=q_* \circ p^* \colon \dCat(Y) \to \dCat([X/G]) .
    \end{equation}
    Suppose that the fiber product
    \begin{equation*}
        Y \times_{X/G} Y = \left\{ (y_1,y_2) \in Y \times Y \mid \tau(y_1) = \tau(y_2)  \right\} \subset Y \times Y
    \end{equation*}
    has dimension $\leq n+1$.
    Then $\tau$ is a crepant resolution and $\Phi$ is an equivalence of categories.
\end{theorem}
It is conjectured that the equivalence \eqref{eq:Bondal--Orlov} holds for any crepant resolution with Gorenstein singularities, cf. \cite{BO02}.

\begin{example}\label{BKR free}
    One can do the reality check: Suppose that the finite group action is free. Then $\omega_X$ is automatically a $G$-sheaf and $\tau$ is an isomorphism. Therefore, $\Phi$ is an equivalence of categories where the equivalence is given by
    \begin{align*}
        \dCat(Y) &\to \dCat([X/G]),  \\
        E & \mapsto (\pi^*E, \lambda_\nat),
    \end{align*}
    which coincides with the description in \eqref{eq:equivalence naive}.
\end{example}

A useful situation is the following case.

\begin{corollary}[{\cite[Corollary~1.3]{BKR01}}]
    Suppose $X$ is a holomorphic symplectic variety and $G$ is a finite group acting faithfully on $X$ by symplectic automorphisms. Assume that $Y$ is a crepant resolution of $X/G$. Then $\Phi$ is an equivalence of categories.
\end{corollary}

A special yet highly non-trivial case is the following result.
\begin{theorem}[{\cite[Theorem~5.1]{Hai01}}]\label{BKR Haiman}
    Let $S$ be a smooth surface, with the symmetric group $\mathfrak{S}_n$ acting on $S^n$ by permutation.
    Then $\Hilb^{\fS_n}_\can(S^n) \simeq S^{[n]}$ is indeed a crepant resolution of $S^n / \fS_n$. 
    Together with the previous corollary, if $S$ is a holomorphic symplectic surface, then
    \begin{equation*}
        \dCat (S^{[n]}) \simeq \dCat([S^n /\fS_n ]).
    \end{equation*}
\end{theorem}

We state the following easy fact for later use.
\begin{lemma}\label{exterior tensor product}
    Suppose for each $i$ we have an equivalence $\dCat(X_i)^{G_i} \simeq \dCat(Y_i)$ given by Fourier--Mukai transform.
    Then we have an equivalence
    \begin{equation*}
        \dCat(X_1 \times \dots \times X_n)^{G_1 \times \dots \times G_n} \simeq \dCat(Y_1 \times \dots \times Y_n).
    \end{equation*}
    given by Fourier--Mukai transform as well.
\end{lemma}
\begin{proof}
    The exterior tensor product gives the Fourier--Mukai kernel and its inverse.
\end{proof}

    



\subsection{Inducing stability conditions} We review the relation between invariant stability conditions and stability conditions on equivariant categories.
For our purpose, we are interested in the following situation.
Let $X$ be a smooth projective variety and $\Lambda$ a free abelian group of finite rank. Let $G$ be a finite group acting on $X$ and $\Lambda$, and let $v\colon \K(X) \to \Lambda$ be a $G$-equivariant homomorphism. This induces a natural $G$-action on $\Stab_\Lambda(X)$, i.e. $g \in G$ acts on $\sigma=(\cP,Z)\in \Stab_\Lambda(X)$ by
\begin{equation*}
    g \cdot (\cP, Z) = (g_* \cP, Z \circ g^*).
\end{equation*}
In particular, one can consider the fixed locus
\begin{equation*}
    \Stab_\Lambda(X)^G\defeq \{ \sigma \in \Stab_\Lambda(X) \mid g \cdot \sigma= \sigma \text{ for all }g\in G \}.
\end{equation*}
This locus is related to $\Stab_\Lambda([X/G])$ as follows. Let $\varpi \colon X \to [X/G]$ be the quotient morphism.
Note that one has a natural functor
    \begin{equation*}
        \varpi^* \colon \dCat([X/G]) \to \dCat(X)
    \end{equation*}
which is just forgetting the linearization in the language of equivariant categories.
Its inverse image provides
\begin{align*}
    (\varpi^*)^{-1}\colon \Stab_\Lambda(X)^G &\to \Stab_\Lambda([X/G]) \\
    (\cP,Z) &\mapsto \left( (\varpi^*)^{-1}\cP, Z\circ \varpi^* \right).
\end{align*}
One can also describe the image of $(\varpi^*)^{-1}$. 
In the language of equivariant categories, any stability condition on $\dCat([X/G])$ that is pulled back from $\Stab_\Lambda(X)^G$ assigns the same phase to objects that differ only by their $G$-linearization.
More precisely, it can be characterized as those satisfying $\cP(\phi)\otimes_G \CC[G] \subset \cP(\phi)$ for all $\phi \in \RR$, where $\CC[G]$ is the regular representation of $G$, cf. {\cite[Proposition 2.2.3]{Pol07}}. 
This correspondence is not only set-theoretic, but also at the level of complex manifolds, by the argument of \cite[Theorem 1.1]{MMS09}.
In summary, we can formally state the results as follows.
\begin{theorem}[{\cite{Pol07,MMS09}}]\label{inducing stability conditions}
    Let $X$ be a smooth projective variety equipped with the action of a finite group $G$. 
    Then there is a one-to-one correspondence
    \begin{equation*}
        \begin{tikzcd}
        \Stab_\Lambda(X)^G \arrow[rr, "(\varpi^*)^{-1}", bend left] &  & {\Stab_\Lambda([X/G])^{[G]}} \arrow[ll, "(\varpi_*)^{-1}", bend left] ,
        \end{tikzcd}
    \end{equation*}
    where
    \begin{equation*}
        \Stab_\Lambda([X/G])^{[G]} \defeq \big\{  \sigma=(\cP,Z) \in \Stab_\Lambda([X/G]) \mid \cP(\phi) \otimes \varpi_*\cO_X \subset \cP(\phi) \text{ for all }\phi\in\RR     \big\}.
    \end{equation*}
    Moreover, both $\Stab_\Lambda(X)^G$ and $\Stab_\Lambda([X/G])^{[G]}$ are unions of connected components in their respective ambient spaces, $\Stab_\Lambda(X)$ and $\Stab_\Lambda([X/G])$.
    Hence, the bijection is actually a biholomorphism.
\end{theorem}
If the action $G$ is good enough to make \eqref{eq:Bondal--Orlov} an equivalence, then we can have a more concrete description of ${\Stab_\Lambda([X/G])^{[G]}}$ from the point of view of $Y$. Here we only touch a special case.
\begin{corollary}\label{co-invariance description}
    Suppose the action $G$ is free, so $Y=X/G$ is smooth and \eqref{eq:equivalence naive} is an equivalence. Then $\sigma \in \Stab_\Lambda(Y)$ lifts to $\Stab_\Lambda(X)$ if and only if $\cP(\phi)\otimes \pi_* \cO_X \subset \cP(\phi)$
\end{corollary}


We can say more about ${\Stab_\Lambda([X/G])^{[G]}}$ if $G$ is a free abelian quotient. 
Let $G^\vee=\Hom(G,\CC^*)$ be the dual group.
Following \cite[Section 4]{Ela14}, there is a natural action of $G^\vee$ on $\cC^G$ by twisting: for any $\chi \in G^\vee$, its action is given by
\begin{equation}\label{eq:residual action}
    \chi.(F,(\lambda_g)) \defeq (F, (\lambda_g)) \otimes \chi \defeq (F, (\lambda_g \cdot \chi(g))).
\end{equation}
In the situation where $G$ is abelian,
\cite[Theorem~4.2]{Ela14} says the category $\cC$ together with its $G$-action can be reconstructed from $\cC^G$ with its $G^\vee$-action: there is an equivalence
\begin{equation*}
    (\cC^G)^{G^\vee}\simeq \cC
\end{equation*}
under which the $G$-action on $\cC$ is identified with the $G \simeq (G^\vee)^\vee$-action.
This suggests the following improvement of Theorem~\ref{inducing stability conditions}.
\begin{proposition}[{\cite{Del23}}]\label{inducing stability conditions abelian}
    Let $G$ be a finite abelian group acting freely on a smooth projective variety $X$. Let $\pi \colon X \to Y \defeq X /G$ be the quotient map. Consider the action of $G^\vee$ on $\dCat(X)^G \simeq \dCat(Y)$ as in \eqref{eq:residual action}. Then there is a one-to-one correspondence between $G$-invariant stability conditions on $\dCat(X)$ and $G^\vee$-invariant stability conditions on $\dCat(Y)$ which preserves geometric stability conditions:
    \begin{equation*}
        \begin{tikzcd}
\Stab(X)^G \arrow[rr, "(\pi^*)^{-1}", bend left] &  & \Stab(Y)^{G^\vee} \arrow[ll, "(\pi_*)^{-1}", bend left] .
\end{tikzcd}
    \end{equation*}
\end{proposition}

\begin{corollary}\label{co-invariance description abelian}
    Let the setting be as in Corollary~\ref{co-invariance description}. Suppose moreover that $G$ is abelian and the action is free. Then checking if $\sigma \in \Stab_\Lambda(Y)$ lifts to $\Stab_\Lambda(X)$ is equivalent to checking $\cP(\phi)\otimes L = \cP(\phi)$ for certain line bundles with $\pi^*(L) \simeq \cO_X $. Moreover, this lift preserves geometric components.
\end{corollary}
\begin{proof}
    This is from the construction and \eqref{eq:residual action}. 
\end{proof}

One can do the reality check in the following special case, which is also a motivating example of Proposition~\ref{inducing stability conditions abelian}.
\begin{example}[Serre invariance \cite{PPZ23}]\label{serre invariance}
    An important example is $\pi \colon X \to Y$ where $X$ is a K3 surface and $Y$ is an Enriques surface (or more generally, $X$ is a hyper-K\"{a}hler manifold and $Y$ is an Enriques manifold). 
    Then the covering involution $\tau \colon X \to X$ induces a dual action on $\dCat(Y)$, whose generator is tensoring with $\omega_Y$ by construction,
    as one can check this directly from \eqref{eq:residual action}.
    As a consequence, if a stability condition on $Y$ is $- \otimes \omega_Y$-invariant, then it induces a stability condition on $X$ (which is also involution-invariant).
    
\end{example}

\section{Constructions on products}\label{sec:product}
In this section we review the constructions of constant $t$-structures and stability conditions on product spaces.
For later use in Section~\ref{sec:restrict}, we need an adaptation of the work \cite{AP06,Pol07,Liu21} to the local version $\dCat_Y(X)$.
It is probably well-known to experts, but for completeness and rigorousness, we state and prove the local version in the section, following the original proofs.


\begin{notation}
    Throughout this section, we let $X,Y$ be smooth projective varieties with $Y \subset X$, $S$ be a smooth variety, and let $p \colon X \times S  \to X$ and $q \colon X \times S \to S$ be the two projections. 
    We will heavily study `sheaves of $t$-structures on $\dCat_Y(X)$ over some base'. For brevity, we will write $\cD \defeq \dCat_Y(X)$ and $\cD_S \defeq \dCat_{Y \times S}(X \times S)$.
    Moreover, if there is a canonical $t$-structure on $\cD_S$ (for example, if it comes from a `sheaf of $t$-structures' over a larger base, cf. Definition~\ref{S-local} and Proposition~\ref{t-structure on closed subvariety}; particularly if it comes from the construction in Proposition~\ref{constant t-structure}), we will write $\cA_S \subset \cD_S$ for the heart of this $t$-structure.
\end{notation}

\subsection{Local t-structures over base}

Following \cite[Definition~2.1.1]{AP06} and \cite{Pol07}, we introduce the following definition.
\begin{definition}\label{S-local}
    A $t$-structure $\left( \cD_S^{\leq 0} ,\cD_S^{\geq 0} \right)$ on $\cD_S$ is called \emph{$S$-local} if for every open subset $U\subset S$, there exists a $t$-structure $\left( \cD_U^{\leq 0} ,\cD_U^{\geq 0} \right)$ on $\cD_U$, such that the restriction functor $\cD_S \to \cD_U$ is $t$-exact.
\end{definition}

In other words, it is $S$-local if it extends to a collection of compatible $t$-structures for all opens subsets. This collection is called a `sheaf of $t$-structures (on $\cD$) over $S$' especially if each $t$-structure is bounded.

Following \cite[Theorem~2.1.4]{AP06} and \cite[Theorem~2.3.2]{Pol07}, we have the following criterion to check locality.
\begin{proposition}\label{check S-local}
    Let $L$ be an ample line bundle on $S$.
    Then a bounded $t$-structure
    $\left( \cD_S^{\leq 0} ,\cD_S^{\geq 0} \right)$  on $\cD_S$ is $S$-local if and only if tensoring with $ q^*L$ is (left) $t$-exact.
\end{proposition}

\begin{lemma}
    Let $f_1,\dots,f_n$ be sections of some line bundle $L$ on $X$ such that $Z$ is the set of common zeroes of $f_1,\dots,f_n$. Let $E \in \dCat_{Y }(X )$.
    Then $E \in \dCat_{Y\cap Z}(X)$ if and only if the morphisms
    \begin{equation*}
        f_{i_1}\dots f_{i_d} \colon E \to E \otimes L^d
    \end{equation*}
    are zero for all sequences $(i_1,\dots, i_d)$ of length $d$ for some $d>0$.
\end{lemma}
\begin{proof}
    This is by \cite[Lemma~2.1.5]{AP06} and $\dCat_Y(X) \cap \dCat_Z(X) = \dCat_{Y \cap Z}(X)$.
\end{proof}
As a consequence, support can be checked via cohomology under certain condition.
\begin{lemma}\label{check supp on coho}
    Let $\left( \cD_S^{\leq 0} ,\cD_S^{\geq 0} \right)$ be a bounded $t$-structure on $\cD_S$ and let $\cH^i$, $i\in \ZZ$ be the associated cohomology functors. 
    Assume that there exists some ample line bundle $L$ on $S$ such that tensoring with $q^*L$ is (left) $t$-exact.
    Then for every closed subset $T\subset S$ one has $E \in \cD_T$ if and only if $\cH^i(E) \in \cD_T$ for all $i$.
\end{lemma}
\begin{proof}
    The ``if'' part is because the subcategory is saturated (thick) and our $t$-structure is bounded. The ``only if '' part follows from the previous lemma. Indeed, assume the vanishing of the morphism of multiplication by $s$
    \begin{equation*}
        E \xrightarrow{s} E \otimes q^*L.
    \end{equation*}
    Taking the truncation functor $\tau^{\geq 0}$ (resp. $\tau^{\leq 0}$) and since tensoring with $q^*L$ is $t$-exact, we know
    \begin{equation*}
        \tau^{\geq 0} E \to \tau^{\geq 0} (E \otimes L^d) = \tau^{\geq 0} E \otimes L^d
    \end{equation*}
    (resp. $\tau^{\leq 0} E \to \tau^{\leq 0} E \otimes L^d$) is zero.
    (In fact, by the proof of \cite[Theorem~2.3.2]{Pol07} only left $t$-exactness is enough.)
\end{proof}

\begin{proof}[Proof of Proposition~\ref{check S-local}]
    This follows from the proof \cite{AP06} and we sketch it for completeness. 
    The `only if' part follows from the same argument as in \cite[Lemma~2.1.2 and Proposition~2.1.3]{AP06}.
    Now we prove the `if' part.
    Let $j \colon U \to S$ be the open immersion of a Zariski open set, and let $T$ be the complement of $U$.
    Note that the restriction functor $\cD_S \to \cD_U$ is essentially surjective. This means $\cD_U$ is the localization of $\cD_S$ with respect to the localizing class consisting of morphisms
    \begin{equation*}
        \{f \colon A' \to A \mid \cone(f) \in  \cD_T \}.
    \end{equation*}
    So to restrict a $t$-structure, it suffices to check all semi-orthogonality is respected under the localization. That is to say, if there is a roof
    \begin{equation}\label{eq:roof}
        \begin{tikzcd}
        A  & A' \arrow[l,"f"'] \arrow[r] &  B
        \end{tikzcd}
    \end{equation}
    with $A \in \cD_U^{\leq 0}$, $B \in \cD_U^{ \geq 1}$, and $C \defeq \cone(f) \in \cD_T$, then it is equivalent to the zero map.
    Note that the long exact sequence associated to $A' \to A \to C \to A'[1]$ implies
    \begin{center}
        $\cH^0 C \xdhrightarrow{} \cH^1 A'$ \quad and \quad $\cH^i C \xrightarrow{\sim} \cH^{i+1} A'$ for $i \geq 1$.
    \end{center}
    By Lemma~\ref{check supp on coho}, whether $A'$ and $C$ are in $\cD_T$ is equivalent to checking on cohomology (of this $t$-structure), so we have 
    \begin{equation*}
        \cone(\tau^{\leq 0} A' \xrightarrow{g} A') =\tau^{\geq 1}A'  \in \cD_T,
    \end{equation*}
    since $\cH^i(A') \in \cD_T$ for any $i \geq 1$.
    As a consequence, \eqref{eq:roof} is equivalent to
    \begin{equation*}
        \begin{tikzcd}
        A  & \tau^{\leq 0} A' \arrow[l,"f\circ g"'] \arrow[r] & B.  
        \end{tikzcd}
    \end{equation*}
    But $\Hom(\tau^{\leq 0}A',B)=0$ which proves the claim.
\end{proof}

A sheaf of $t$-structures over $S$ also restricts to a smooth closed subvariety $T \subset S$.
\begin{proposition}\label{t-structure on closed subvariety}
    Let $T\subset S$ be a closed embedding of smooth projective varieties.
    Then for every $t$-structure on $\cD_S$ local over $S$, there exists a unique $t$-structure on $\cD_T$ local over $T$ such that the pushforward functor $i_* \colon \cD_T \to \cD_S$ is $t$-exact.
\end{proposition}
\begin{proof}
    It is unique and bounded if it exists, as it has to take the form
    \begin{equation*}
        \cD_T^{[a,b]} = \{ F \mid i_*F \in \cD_S^{[a,b]} \}
    \end{equation*}
    by Remark~\ref{conservatively t-exact}.
    So we only need to show this indeed defines a $t$-structure on $\cD_T$.
    
    First, we check for $A \in \cD_T^{\leq 0}$, $B \in \cD_T^{\geq 1}$ one has $\Hom(A,B)=0$. 
    Since $A,B \in \dCat(X \times S)$, we can apply \cite[Lemma~2.1.10]{AP06}, which says
    that it suffices to prove that 
    \begin{equation*}
        q_* \RcHom(A,B) \in D(T)^{\geq 1} 
    \end{equation*}
    (with respect to the standard $t$-structure).
    Since the standard $t$-structure is local, we can check it locally, so we may assume that $T$ is affine.
    So it suffices to check $\Hom^i(A,B)$ for $i\leq 0$, which follows from \cite[Lemma~2.5.1]{AP06}.

    It remains to check $\cD_T$ is generated by $\cA_T \defeq \cD^{\leq 0} \cap \cD^{\geq 0}$.
    Let $F \in \cD_T$ be an arbitrary object. By shifting we may assume $i_*F \in \cD_S^{\geq 0}$ but $i_*F \notin \cD_S^{\geq 1}$. Let $\phi \colon \tau^{\leq 0} i_*F \to i_*F$ be the canonical map.
    It suffices to prove there exists an object $F^0 \in \cD_T$ and a morphism $\bar\phi$ such that 
    \begin{equation}\label{eq:truncation}
        \tau^{\leq 0} i_*F \simeq i_* F^0
    \end{equation}
    and $\phi = i_* \bar\phi$. Indeed, then we would consider the cone $F'$ of the morphism $\bar\phi$, and since $i_*F'$ has smaller cohomological length than $i_*F$ with respect to our $t$-structure, we would be able to finish the proof by induction.

    By the proof of \cite[Theorem~2.5.2]{AP06}, we already find such $F^0 \in \dCat(X \times T)$, and by definition $\tau^{\leq 0} i_*F \in \cD_S$, so from \eqref{eq:truncation} we see $i_*F^0 \in \dCat_{Y\times T}(X \times S)$, hence $F^0 \in \cD_T$.
\end{proof}

\subsection{Constant $t$-structures on products}
In this section, we follow \cite{AP06,Bayer-Craw-Zhang:2017} to construct a $S$-local t-structure on $\cD_S$. Recall that a $t$-structure is called noetherian if its heart is.

\begin{proposition}\label{constant t-structure}
    Given a noetherian bounded $t$-structure $(\cD^{\leq 0},\cD^{\geq 0})$ on $\cD$,
    one can construct a noetherian bounded $t$-structure $(\cD_S^{\leq 0},\cD_S^{\geq 0})$ on $\cD_S$ which is local over $S$, for any quasi-projective variety $S$.
    If $S$ is projective, its heart is given by
    \begin{equation*}
        \cA_S= \left\{ E \in \cD_S \mid p_*(E \otimes q^*L^n) \in \cA \text{ for } n \gg 0 \right\},
    \end{equation*}
    where $L$ as any ample line bundle on $S$.
\end{proposition}

We will call the heart $\cA_S$ the Abramovich--Polishchuk (AP) heart on $\cD_S$ in the sequel.

\begin{proof}
    We follow the same proof of \cite{AP06}.

    By Proposition~\ref{check S-local}, this $t$-structure is $S$-local from the very construction.
    So we only need to show that this is a bounded $t$-structure.
    \subsubsection*{Step 1}
    Firstly, we claim
    \begin{equation}\label{eq:glued t-structure}
        (\cD_{\PP^r})_0^{[a,b]} \defeq \{ E \in \cD_{\PP^r} \mid p_*E \in \cD^{[a,b]}, \dots, p_*E(r) \in \cD^{[a,b]}  \}
    \end{equation}
    defines a bounded $t$-structure on $\cD_{\PP^r}$.
    Indeed, the standard semiorthogonal decomposition $\dCat(\PP^r) \simeq \langle \cO(-r), \dots ,\cO \rangle$ pulls back to a semiorthogonal decomposition
    \begin{equation*}
        \dCat(X \times \PP^r) = \langle p^* \dCat(X)(-r), \dots, p^* \dCat(X) \rangle
    \end{equation*}
    and restricts to
    \begin{equation*}
        \cD_{\PP^r} = \langle p^* \cD(-r), \dots, p^*\cD \rangle.
    \end{equation*}
    Then \eqref{eq:glued t-structure} is just obtained by gluing from the $t$-structure on every piece $\cD$ along this semiorthogonal decomposition, cf. the proof of \cite[Proposition~2.3.1]{AP06}.
    \subsubsection*{Step 2}
    Via the modified semiorthogonal decomposition $\dCat(\PP^r) = \langle \cO(-n-r) , \dots , \cO(-n)\rangle$, we can define
    \begin{equation*}
        (\cD_{\PP^r})^{[a,b]}_n \defeq \{ E \in \cD_{\PP^r} \mid p_*E(n) \in \cD^{[a,b]}, \dots, p_*E(n+r) \in \cD^{[a,b]}  \}
    \end{equation*}
    and easy to see
    \begin{equation*}
        (\cD_{\PP^r})^{[0,0]}_0 \subset (\cD_{\PP^r})^{[-r,0]}_n
    \end{equation*}
    for $n \geq 0$, following the same proof of \cite[Corollary~3.2.4]{Bayer-Craw-Zhang:2017}.
    As a consequence, we have inclusions
    \begin{equation}\label{eq:inclusions}
        \begin{aligned}
            (\cD_{\PP^r})^{\leq 0}_0 \subset (\cD_{\PP^r})^{\leq 0}_1 \subset (\cD_{\PP^r})^{\leq 0}_2 \subset \dots \subset (\cD_{\PP^r})^{\leq r}_0,\\
            (\cD_{\PP^r})^{\geq 0}_0 \supset (\cD_{\PP^r})^{\geq 0}_1 \supset (\cD_{\PP^r})^{\geq 0}_2 \supset \dots \supset (\cD_{\PP^r})^{\geq r}_0.
        \end{aligned}
    \end{equation}
    In particular, there is a natural map $\cH^k_n(-) \to \cH^i_{n+1}(-)$.
    \subsubsection*{Step 3}
    By the construction, we know
    \begin{equation*}
        \cD_{\PP^r}^{\leq 0} = \bigcup_n (\cD_{\PP^r})^{\leq 0}_n \quad \text{and} \quad \cD_{\PP^r}^{\geq 0} = \bigcap_n (\cD_{\PP^r})^{\geq 0}_n .
    \end{equation*}
    Hence we defined a bounded $t$-structure on $\cD_{\PP^r}$.
    By construction the semiorthogonality is automatic. It remains to show that they generate the whole category.
    It suffices to show that the cohomology will stabilize as we can define the truncation functor as $n$ large, i.e. $\cH^k_n(E) \simeq \cH^i_{n+1}(E)$ for any $E$.
    By boundedness, we assume that $E \in (\cD_{\PP^r})^{[a,b]}_0$ and by \eqref{eq:inclusions} we know that $E \in (\cD_{\PP^r})^{[a-r,b]}_n$. If we can show $\cH^b_n(E) \simeq \cH^b_{n+1}(E)$ for $n \gg 0$, then we can conclude by induction on the length of the interval $[a-r,b]$.
    By the same reduction as in proof of \cite[Proposition~3.4.2]{Bayer-Craw-Zhang:2017}, it suffices to show: for $E \in (\cD_{\PP^r})^{[0,0]}_0$, we have $p_* \left( \cH^0_0(E(n))(i+1)\right) \in \cD^{\geq 0}$ for $0 \leq i \leq r$ and $n \gg 0$.
    But this follows from the same proof of \cite[Lemma~3.4.1]{Bayer-Craw-Zhang:2017}.

    \subsubsection*{Step 4}
    Then we restrict it to a closed subset of $\PP^r$ using Proposition~\ref{t-structure on closed subvariety} and furthermore its open subset using Proposition~\ref{check S-local}.
    \subsubsection*{Step 5}
    Following a similar analysis in the proof of \cite[Theorem 2.3.6]{AP06} we can show $\cA_S$ is noetherian.
\end{proof}

\begin{remark}
    Following the same proof we can see that this is independent of the choice of $L$. 
\end{remark}

\subsection{Open heart property}

In this section, we follow \cite[Section~3]{AP06} to show that every $S$-local $t$-structure on $\cD_S$ enjoys the `open heart property'.

\begin{proposition}[Open heart property]\label{open heart property}
    Let $\left( \cD_S^{\leq 0} ,\cD_S^{\geq 0} \right)$ be a noetherian $t$-structure on $\cD_S$ that is local over $S$.
    Let $E\in \cD_S$ and let $T\subset S$ be a smooth closed subscheme such that $i_T^* E \in \cA_T$. Then there exists a Zariski open neighborhood $U \subset S$ of $T$ such that $E_U \in \cA_U$.
\end{proposition}

The proof relies on the following two lemmas.

\begin{lemma}\label{noether implies}
    Assume that $\cA_S$ is noetherian and $T \subset S$ is a smooth divisor.
    Then for any $E \in \cA_S$, there exists an $S$-torsion-free object $\overline{E} \in \cA_S$ such that $i^* \overline{E} \in \cA_T$
\end{lemma}
\begin{proof}
    We are following \cite[Proposition~3.1.2 and Corollary 3.1.3]{AP06}.
    Consider the exact sequence in $\cA_S$
    \begin{equation*}
        0 \to F \to E \to \overline{E} \to 0
    \end{equation*}
    where $F$ is the maximal $S$-torsion subobject supported over $T$, which exists due to the noetherian hypothesis. 
    By construction, $\overline{E}$ contains no $S$-torsion subobject supported over $T$.
    In particular, the natural map $\overline{E}(-T) \xrightarrow{f} \overline{E}$ is injective, as its kernel is supported over $T$.
    So we have $i_T^* \overline{E} \in \cA_T$.
\end{proof}

\begin{lemma}\label{open heart lemma}
    Let $S$ be a smooth quasi-projective scheme. Let $E \in \cA_S$, let $T \subset S$ be a smooth subscheme, and assume that $\cH^0(i_T^* E)=0$.
    Then there is an open neighborhood $T \subset U \subset S$ such that $E_U =0$.
\end{lemma}
\begin{proof}
    We first consider the case where $T$ is a divisor in $S$, defined by $f=0$.

    Since $i^*$ is right $t$-exact, we have a surjection $\cH^0(i_T^*E) \xdhrightarrow{} \cH^0(i_T^*\overline E)$.
    By assumption and Lemma~\ref{noether implies}, we know $i_T^* \overline{E} =0$.
    Consequently, $\overline{E}$ is supported away from $T$, thus it vanishes on some open neighborhood of $T$.

    So we may assume that $E=F$ is supported over $T$. Suppose that $T\subset S$ is defined by $f$ and $E$ is annihilated by $f^k$ for some minimal integer $k\geq 0$.
    Note that
    $i_{T*} i_T^* E = \cone\left(E(-T) \xrightarrow{f}E \right)$, 
    and therefore
    \begin{equation*}
        i_{T*} \cH^0 (i_T^* E) = \cH^0 (i_{T*}i_T^* E) = E / f E(-T).
    \end{equation*}
    By assumption $\cH^0 i_T^* E=0$, we have $E=f E(-T)$ and therefore $E$ is annihilated by $f^{k-1}$, contradicting minimality.

    For the general case, we do induction on the codimension $\dim S - \dim T >1$.
    Assume the result holds for lower codimensions.
    Write
    \begin{equation*}
        T \xhookrightarrow{ \ i_{T \subset S_1} \ } S_1 \xhookrightarrow{ \  i_{S_1} \ } S
    \end{equation*}
    where $S_1 \subset S$ is a smooth divisor. Then
    \begin{equation*}
        \cH^0(i_T^* E) = \cH^0 \left( i_{T \subset S_1}^* \cH^0 (i_{S_1}^* E) \right) .
    \end{equation*}
    By induction we have that there is an open neighborhood $T \subset U_1 \subset S_1$ such that
    \begin{equation*}
        \left( \cH^0(i_{S_1}^* E) \right)_{U_1} =0 .
    \end{equation*}
    By replacing $S$ by an open subset containing $U_1$ as a divisor, we are reduced to the divisor case and get that $E_U =0$, as desired.
\end{proof}

\begin{proof}[Proof of Proposition~\ref{open heart property}]
    By induction on the codimension as in the lemma, it suffices to prove for the case where $T$ is a divisor.
    If we know $\tau^{>0} E =0$ then we are done.
    Indeed, taking the long exact sequence of cohomology of the distinguished triangle 
    $i_T^*(\tau^{<0}E) \to i_T^* E \to i_T^*(\cH^0 E)$ we get
    \begin{equation*}
        \dots \to \cH^{k} \left( i_T^*(\tau^{<0}E) \right) \to \cH^{k} \left( i_T^* E \right) \to \cH^{k} \left( i_T^*(\cH^0 E) \right) \to \dots
    \end{equation*}
    By right exactness of $i_T^*$ and left exactness of $i_T^*[-1]$, we deduce 
    $\cH^{k} \left( i_T^*(\tau^{<0}E) \right) =0$ for all $k$. By Lemma~\ref{open heart lemma} again, one may shrink $U$ to make $\tau^{<0}E_U=0$, hence $E_U= \cH^0 E_U$.

    Now set $M$ to be the largest integer with $\tau^{\geq M} E \neq 0$ and thus $\tau^{\geq M} E = \cH^M(E)[-M]$. It suffices to show $M\leq 0$. 
    Assume by contradiction $M>0$. Since $i_T^*$ is right exact, we get that
    \begin{equation*}
        \cH^0(i_T^*\cH^M E) = \cH^M(i_T^*(\tau^{\geq M} E)) = \cH^M(i_T^*E) =0 .
    \end{equation*}
    So by Lemma~\ref{open heart lemma} we know $\cH^M(E_U)=0$, which contradicts the definition of $M$.
\end{proof}

We say that an algebraic group $G$ acts on $\dCat_Y(X)$ by Fourier--Mukai transforms of $\dCat(X)$ if we are given a homomorphism $G \to \Aut \dCat(X)$, $g \mapsto \Phi_g$ such that every $\Phi_g$ preserves $\dCat_Y(X)$, and an object $K \in \dCat(G \times X \times X)$ of finite Tor-dimension with the support proper over $G \times X$ (with respect to the projection $p_{13}$), such that for every $g \in G(\CC)$ we have $\Phi_g =\Phi_{K|_{g \times X \times X}}$.

\begin{example}
    Let $X$ be smooth projective and let $G$ be a subgroup of $\Aut(X)$ preserving $Y$. Then we have the map
    \begin{equation*}
        \act \colon G \times X \to X
    \end{equation*}
    and denote its graph by $\Gamma_\act$.
    This defines an action of $G$ on $\dCat_Y(X)$ by Fourier--Mukai transforms, by setting $K=\cO_{\Gamma_\act}$, as $\Gamma_\act$ is closed in $G \times X \times X$.
\end{example}

An application is the following theorem
\begin{theorem}\label{invariant action}
    Let $G$ be a connected smooth algebraic group acting on $\dCat_Y(X)$ by Fourier--Mukai transforms of $\dCat(X)$ and let $(\dCat_Y(X)^{\leq 0},\dCat_Y(X)^{\geq 0})$ be a bounded noetherian t-structure on $\dCat_Y(X)$. Then the $t$-structure is invariant under this group action.
    In particular, if we have a continuous family of stability conditions such that there is a dense subset with noetherian heart, then any stability is invariant under the action $G$.
\end{theorem}
\begin{proof}
    The proof is exactly the same as the proof of \cite[Theorem~3.5.1]{Pol07}. We repeat it here for completeness.
    Consider the universal action
    \begin{equation*}
        K * E \defeq p_{13*}(K \otimes p_2^* E ).
    \end{equation*}
    Then $\Phi_g(E) \simeq (K * E) |_{g \times X}$.
    Note that $(K * E)|_{e \times X} \simeq E \in \cA$, so by Proposition~\ref{open heart property} we know that there exists an open neighborhood $U$ of $e$ in $G$ such that $(K * E)|_{U \times X} \in \cA_U$. Since restriction to closed subscheme is right $t$-exact (as it admits the right adjoint push-forward functor which is $t$-exact), we know $(K * E)|_{g \times X} \in \cD^{\leq 0}$ and hence $(K * E)|_{g^{-1} \times X} \in \cD^{\geq 0}$.
    It follows that for $g \in U \cap U^{-1}$ the functor $\Phi_g$ is $t$-exact. Therefore, the set of $g$ such that $\Phi_g$ is $t$-exact is an open (and closed) subgroup in $G$, so it is equal to $G$.
\end{proof}

\subsection{Stability conditions on products with curves}
In this section, we summarize the main result in \cite{Liu21}, which constructs stability conditions on $\dCat(X\times C)$ from any given stability condition on $\dCat(X)$ where $C$ is a smooth projective curve. More precisely, it constructs a family of stability conditions on the tilted AP heart of $\dCat(X \times C)$ from a stability condition on $X$ if the base $C$ is a curve.

We start from $\sigma=(\cA,Z)$  a stability condition on $X$ satisfying support property with respect to $\K(X) \xdhrightarrow{v} \Lambda$. 
For simplicity, we also assume that $\cA$ is noetherian and $Z$ is rational, i.e. takes values in $\QQ \oplus \I \QQ$.
By the proof of \cite[Theorem~3.3]{Liu21}, we know that
\begin{equation}\label{eq:Z-polynomial}
    Z \left(p_*(E \otimes q^* L^n) \right) 
\end{equation}
is a polynomial in $n$ of degree at most $\dim(S)$.
In the case where $S=C$ is a curve, \eqref{eq:Z-polynomial} is linear in $n$, and we write \eqref{eq:Z-polynomial} as
\begin{equation}\label{eq:def abcd}
    \va n+ \vb + \I (\vc n + \vd)
\end{equation}
where $a,b,c,d$ are group homomorphisms from $\K(X \times C)$ to $\RR$.
Note that if we assume $Z$ takes values in $\QQ \oplus \I \QQ$, then $a,b,c,d$ take values in $\QQ$.
It is easy to check that these homomorphisms satisfy certain inequalities described in \cite[Lemma~4.1]{Liu21},
for example, for any $E$ in the AP heart $\cA_C$, we have
\begin{equation}\label{eq:abcd}
    c(E) \geq 0, \text{ and if $c(E)=0$, then $d(E)\geq 0$ and $a(E) \leq 0$}.
\end{equation}
In particular, for any $t \in \RR_{>0}$ and $\beta \in \RR$, 
\begin{equation*}
    Z_C^{t,\beta}(E)=t \va -\vd +\beta \vc  +i t \vc ,
\end{equation*}
defines a weak stability function on $\cA_C$, whose slope function
\begin{equation*}
    \nu_{t, \beta}(E)= 
    \left\{
    \begin{array}{ll}
        \frac{-ta+d}{tc}-\beta & \text{if $c\neq 0$,} \\
         +\infty & \text{otherwise}
    \end{array}
    \right .
\end{equation*}
gives rise to a torsion pair of $\cA_C$
\begin{align*}
    \cT^{t,\beta} &\defeq \{ E \in \cA_C \mid  \nu_{t,\beta}^-(E) >0\}, \\
    \cF^{t,\beta} &\defeq \{ E \in \cA_C \mid  \nu_{t,\beta}^+(E) \leq 0\}.
\end{align*}
We construct the tilted heart as the extension-closure
\begin{equation*}
    \cA_C^{t,\beta} \defeq \langle \cT^{t,\beta}, \cF^{t,\beta}[1] \rangle.
\end{equation*}
The main theorem of \cite{Liu21} is the following:
\begin{theorem}\label{stabilities on product} 
    Let
    \begin{equation}\label{eq:central charges on product}
        Z_C^{s,t,\beta}(E)= s \vc +\vb - \beta \va + \I (-t\va +\vd - \beta \vc). 
    \end{equation}
    Then for $s,t \in \RR_{>0}$ and $\beta \in \RR$, $\sigma^{s,t,\beta}\defeq \left( \cA_C^{t,\beta} , Z_C^{s,t,\beta}\right) $ is a stability condition on $\dCat(X \times C)$ satisfying support property with respect to 
    $(v_1,v_2)\colon \K(X \times C) \to \Lambda \oplus  \Lambda / \ker Z$, where
    \begin{align*}
        v_1 (E) &= v(p_*(E \otimes q^*L^n)) - v(p_*(E\otimes L^{n-1})), \\
        v_2 (E) &= v(p_*(E \otimes q^*L^n)) - n \cdot v_1(E)
    \end{align*}
    for $n\gg 0$.
\end{theorem}

\begin{remark}\label{local stability conditions}
    Following the same proof, one can construct a family of stability conditions on $\cD_C$ from a stability condition on $\cD$.
\end{remark}


By applying Theorem~\ref{stabilities on product} inductively, one can construct a family of stability conditions on the product of curves which exhibit certain symmetries.
This following computation appears in \cite{Haiden-Sung:2024} and will also be in \cite{LMPSZ}. We include it for the convenience of the reader.

\begin{proposition}\label{special stability via curve product}
    Let $C_i$ be a smooth projective curve with genus $g_i \geq 1$, and let $\cO_{C_i}(1)$ be an ample line bundle of degree $1$ on $C_i$.
    Then for any $k \geq 1$, there is a continuous family of stability conditions $\sigma_k^{t,b}$ on $C_1 \times \dots \times C_k$ parametrized by $b \in \RR$ and $w \in \RR_{>0}$, such that
    \begin{itemize}
        \item the central charge takes the form
        \begin{equation*}
            Z_k^{w,b}(-)=-\int_{C_1 \times \dots \times C_k} e^{-(b+\I w)H_k} \cdot \ch (-) ,
        \end{equation*}
        where $H_k \defeq c_1 \left(\cO_{C_1}(1)\boxtimes \ldots \boxtimes \cO_{C_k}(1) \right)$;
        \item all skyscraper sheaves are stable of phase $1$.
    \end{itemize}
\end{proposition}
\begin{proof}
    We prove this by induction on the number of curves.
    For the initial case $k=1$, we have
    \begin{equation*}
        Z_1^{w,b}= -\int_{C_1} e^{-(b+\I w)H} \cdot \ch = (- \frac{\ch_1}{H} +b \ch_0 )+ \I w \ch_0.
    \end{equation*}    
    This is the central charge of the stability condition $\sigma_1^{w,b}=\left( \Coh(C_1), -\deg + b \rk + \I w \rk \right)$, and
    it is clear that all skyscraper sheaves are $\sigma_1^{w,b}$-stable of phase $1$.
    
    Assume the statement holds true for $k$ and we want to show for $k+1$. For simplicity we denote by $X_k \defeq C_1 \times \dots \times C_k$. We are in the situation
    \begin{equation*}
        \begin{tikzcd}
        X_k \times C_{k+1} \arrow[d, "p"'] \arrow[r, "q"] & C_{k+1} \\
        X_k                                            &        
        \end{tikzcd}
    \end{equation*}
    and we denote $\cL_{k+1}\defeq q^* \cO_{C_{k+1}}(1)$ and $L_{k+1} \defeq c_1(\cL_{k+1})$ for its numerical class. 
    By the induction hypothesis, we already know the central charge on $X_k$ takes the form
    \begin{equation*}
        Z_k^{w,b}=-\int_{X_k} e^{-(b+\I w)H_k} \cdot \ch .
    \end{equation*}
    Then
    \begin{align*}
        Z_k^{w,b}(p_*(- \otimes \cL_{k+1}^n)) &=-\int_{X_k} e^{-(b+\I w)H_k} \cdot \ch \left(p_* (- \otimes \cL_{k+1}^n) \right) \\
        &=-\int_{X_k} e^{-(b+\I w)H_k} \cdot p_*\left(\ch(- \otimes \cL_{k+1}^n) \cdot \td({X_k} \times C_{k+1}) \right) \cdot \td({X_k})^{-1} \quad \\ 
        &= -\int_{X_k} p_* \left( p^* \left(e^{-(b+\I w)H_k} \td({X_k})^{-1} \right) \cdot \ch (- \otimes \cL_{k+1}) \cdot \td({X_k}\times C_{k+1}) \right) \\
        &= -\int_{{X_k}\times C_{k+1}}  p^* \left(e^{-(b+\I w)H_k} \td({X_k})^{-1} \right) \cdot \ch (- \otimes \cL_{k+1}) \cdot \td({X_k}\times C_{k+1}) \\
        &= -\int_{{X_k}\times C_{k+1}}  p^* e^{-(b+\I w)H_k} \cdot \ch (-) \cdot (1+nL_{k+1}) \cdot \td(p) \\
        &= \left( - \int_{{X_k} \times C_{k+1}} p^* e^{-(b+\I w)H_k} \cdot L_{k+1} \cdot \ch (-) \right )  n \\ 
        &\phantom{=} + \left ( - \int_{{X_k} \times C_{k+1}} p^* e^{-(b+\I w)H_k} \cdot \left(1+ (1-g_{k+1})L_{k+1} \right) \cdot \ch (-) \right) \\
        & \eqdef ( a(-)+\I c(-))n + ( b(-) +\I d(-) ) .
    \end{align*}
    Here we use the same notation $a$, $b$, $c$, and $d$ as in \eqref{eq:abcd}.
    Then by Theorem~\ref{stabilities on product}, we can construct a family of stability condition $(\sigma_k^{w,b})^{s,t,\beta}$ on $X_{k+1} = X_k \times C_{k+1}$, for any $s,t \in \RR_{>0}$ and $\beta \in \RR$.
    In particular, if we take $s=t=w$, the central charge \eqref{eq:central charges on product} takes the form
    \begin{align*}
        {(Z_k^{w,b})}_C^{w,w,\beta} (-) &= w c(-) +b(-) + \beta a(-) + \I(- w a(-) +d(-) + \beta c(-)) \\
        &= (\beta-\I w)\left(a(-) +\I c(-))+ (b(-)+\I d(-) \right) \\
        &= (\beta-\I w)\left ( -\int_{X_{k+1}} p^* e^{-(b+\I w)H_k}  \cdot L_{k+1} \cdot \ch (-) \right) \\
        & \phantom{=} + \left ( - \int_{X_{k+1}} p^* e^{-(b+\I w)H_k}  \cdot \left(1+ (1-g_{k+1})L_{k+1} \right) \cdot \ch (-) \right) \\
        &= - \int_{X_{k+1}} \left( (1-g_{k+1}+\beta-\I w) p^* e^{-(b+\I w)H_k}  \cdot  L_{k+1}
        + p^* e^{-(b+\I w)H_k} \right) \cdot \ch (-)  .
    \end{align*}
    Expanding this, we get
    \begin{equation*}
         -\sum_{i=0}^{k+1} \left( (1-g_{k+1}+\beta -\I w) \frac{(-b-\I w)^{i-1}p^* H_k^{i-1} \cdot L_{k+1}}{(i-1)!} +\frac{(-b-\I w)^i p^* H_k^i}{i!}\right) \ch_{k+1-i}(-).
    \end{equation*}
    Note that 
    \begin{equation*}
        \frac{p^*H_k^{i-1} \cdot L_{k+1}}{(i-1)!}+ \frac{p^* H_k^i}{i!}= \frac{H^i_{k+1}}{i!} .
    \end{equation*}
    So we may set $\beta$ such that $1-g_{k+1}+\beta =-b$, then we have
    \begin{equation*}
        Z_{k+1}^{w, b}\defeq {(Z_k^{w,b})}_C^{w,w,\beta} =-\sum_{i=0}^{k+1} \frac{(-b-\I w)^i H^i_{k+1}}{i!}\ch_{k+1-i} = -\int_{X_{k+1}} e^{-(b+\I w)H_{k+1}} \cdot \ch .
    \end{equation*}
    Moreover, since all $g_i \geq 1$, by \cite[Theorem~1.1]{FLZ22} we know all skyscraper sheaves are $\sigma_{k+1}^{w,b}$-stable; and from the expression, we know that they are of phase $1$; see Example~\ref{extending is aligned} in Section~\ref{section:push stable}.
\end{proof}

\section{Restricting stability conditions}\label{sec:restrict}
In this section, we study the problem of restricting a stability condition on $\dCat(X)$ to a stability condition on $\dCat(Y)$, where $Y$ is a closed smooth subvariety of $X$.
As a first step, we restrict a stability condition on $\dCat(X)$ to $\D_Y(X)$ in Section~\ref{section:restrict set}.
To further restrict to $\dCat(Y)$, we need some intermediate steps which we study in Section~\ref{section:push stable} and \ref{section:complex supp}.
Finally we use these tools to achieve our goal, namely Theorem~\ref{restriction}, in Section~\ref{section:restrict scheme}.

All of the restrictions make use of the following observation.

\begin{lemma}\label{formal restriction}
    Let $\Phi \colon \cC \to \cD$ be a faithful triangulated functor between triangulated categories. 
    Let $\sigma$ be a (pre-)stability condition on $\cD$. 
    If for any $E \in \cC$, the HN filtration of $\Phi(E)$ is induced by some filtration of $E$ in $\cC$ along $\Phi$, then $\sigma=(\cP,Z)$ restricts to a (pre-)stability condition $\Phi^{-1}\sigma$ on $\cC$, by setting
    \begin{equation*}
        \Phi^{-1}\sigma \defeq (\Phi^{-1} \cP , Z\circ [\Phi]),
    \end{equation*}
    where
    \begin{equation*}
        \Phi^{-1} \cP(\phi) \defeq \{E \in \cD \mid \Phi(E) \in \cP(\phi)\}.
    \end{equation*}
    Moreover, if $\sigma$ satisfies the support property with respect to $\K(\cD) \xdhrightarrow{v} \Lambda$,
    then $\Phi^{-1} \sigma$ satisfies the support property with respect to
    $\K(\cC) \xdhrightarrow{} \Lambda_\cC$
    as the coimage component of the composition of maps $\K(\cC) \xrightarrow{[\Phi]} \K(\cD) \xdhrightarrow{v} \Lambda$.
\end{lemma}

The proof is straightforward: semi-orthogonality is respected because $\Phi$ is faithful, and the HN property follows automatically by our assumption; see also \cite[Section~2.2]{MMS09}. 
Note that by Remark~\ref{conservatively t-exact}, this is the only way to make $\Phi$ a $t$-exact functor with respect to all the $t$-structures $(\cP(> \phi), \cP(\leq \phi+1) )$.
\begin{definition}\label{compatibility}
    For any triangulated functor between triangulated categories $\Phi \colon \cC \to \cD$, if $\tau \in \Stab(\cC)$ and $\sigma \in \Stab(\cD)$ are related in the way of Lemma~\ref{formal restriction}, we say that
    $\tau$ is the restriction of $\sigma$, or $\tau$ and $\sigma$ are compatible (along $\Phi$).
\end{definition}

\subsection{Set-theoretic restriction}\label{section:restrict set}
We collect in this section some observations and results from \cite{LMPSZ}.
\begin{proposition}\label{filtration reduction}
    Let $X$ be a closed subvariety of a smooth projective variety $W$ and let $q\colon W \to A$ be a morphism to an abelian variety $A$. Let $z$ be a point in $A$ and let $Y$ be a closed subvariety of $X$ which is a (union of) connected component(s) of $q^{-1}(z) \cap X$.
    Then for any stability condition $\sigma$ on $\dCat_X(W)$ and any object $F$ in $\dCat_Y(W)$, its HN factors lie in $\dCat_Y(W)$ as well.
\end{proposition}
\begin{proof}
    See in the forthcoming paper \cite{LMPSZ}.
\end{proof}

This proposition allows us to do the set-theoretic restriction, due to Lemma~\ref{formal restriction}.

\begin{corollary}\label{pre-restriction}
    Let the notation and hypotheses be as in Proposition~\ref{filtration reduction}. 
    Then any stability condition on $\dCat_X(W)$ restricts to a stability condition on $\dCat_Y(W)$ in the way of Lemma~\ref{formal restriction}.
\end{corollary}
\begin{proof}
    Note that $\dCat_Y(W)$ is a full subcategory of $\dCat_X(W)$. By Lemma~\ref{formal restriction}, it suffices to prove that every HN factor of an object in $\dCat_Y(W)$ still lies in $\dCat_Y(W)$, and this is precisely Proposition~\ref{filtration reduction}.
\end{proof}

We denote the restriction map of $\sigma \in \Stab_\Lambda(X)_W$ to $\Stab_{\Lambda_Y}(Y)_W$ by $-|_{Y|W}$.
By studying the properties of this restriction map, one also obtains the following theorem from \cite{LMPSZ}.



\begin{theorem}\label{pre-restriction to end}
    Let $X \defeq C_1 \times \dots \times C_n$ be a product of curves with genera at least one. Then all numerical stability conditions on $\dCat(X)$ are geometric. Moreover, for any $\K(X) \xdhrightarrow{v} \Lambda$ that factors through $\K(X) \xdhrightarrow{} \K_\num(X)$, the forgetful map
    \begin{align*}
        (\phi_\bullet(\cO_z),\cZ) \colon \Stab_\Lambda(X) &\to \RR \times \Hom(\Lambda,\CC) , \\
        \sigma = (\cP , Z) &\mapsto (\phi_{\cP}(\cO_z),Z)
    \end{align*}
    is injective. Here $\cO_z$ is any skyscraper sheaf, and the map is independent of the choice of point $z\in X$.
\end{theorem}

A special case of Corollary~\ref{pre-restriction} where $W=X$ and $q=\alb_X$ allows us to restrict stability conditions on $\dCat(X)$ to $\dCat_{X_0}(X)$ 
if $X_0$ is an Albanese fiber. More generally, we have the following corollary.

\begin{corollary}\label{pre-restriction'}
    Let $f \colon X \to Y$ be a morphism between smooth projective varieties such that the Albanese morphism of $Y$ is finite. 
    Let $X_0$ be a fiber of $f$.
    Then any stability condition on $\dCat(X)$ restricts to a stability condition on $\dCat_{X_0}(X)$ in the way of Lemma~\ref{formal restriction} along the natural closed immersion $X_0 \hookrightarrow X$.
\end{corollary}
\begin{proof}
    Apply Corollary~\ref{pre-restriction} for $W=X$ and $q= \alb_Y \circ f$, where $\alb_Y \colon Y \to \Alb(Y)$ is the Albanese morphism of $Y$.
\end{proof}

A natural question is, if $X_0$ is smooth, when can we further restrict a stability condition on $\dCat_{X_0}(X)$ to a stability condition on $\dCat(X_0)$? We will achieve a sufficient condition in Section~\ref{section:restrict scheme}.

\subsection{Pushforward of stable objects}\label{section:push stable}

Let $C$ be a smooth projective curve.
Consider the situation where $i\colon X \to X \times C$ is the closed immersion of a fiber of the projection $X \times C \to X$, and suppose that we are given a (rational) stability condition $\sigma \in \Stab(X)$. Recall that by Theorem~\ref{stabilities on product}, there always exists a family of stability conditions on $X \times C$.
Let $F\in \dCat(X)$. We want to relate the stability of $F$ to the stability of $i_* F$ with respect to some stability conditions in this family.
It is natural to guess they coincide at some `large volume limit'.
(For example, one evidence is that this is true for $i$ being the closed immersion of curves on a smooth projective surface, cf. \cite{Mac14}.)
We will show that at least for $C \neq \PP^1$, this holds true for any stability condition on $X \times C$ that comes from Liu's construction in Theorem~\ref{stabilities on product}.

Firstly, we introduce a notion of `numerical compatibility' between stability conditions.

\begin{definition}\label{alignment}
    Let $\Phi \colon \cC \to \cD$ be a faithful triangulated functor. For $\tau \in \Stab(\cC)$ and $\sigma \in \Stab(\cD)$, we say that $\tau$ and $\sigma$ are \emph{numerically compatible} (along $\Phi$) if $Z_{\sigma} \left( \Phi (F) \right)= Z_{\tau}(F)$ for any $F \in \cC$.
\end{definition}

Recall the compatibility defined in Definition~\ref{compatibility}, which requires in addition that the functor respects (semi)stability. More precisely, it is easy to describe their difference in the following equivalent ways. 
\begin{lemma}\label{aligned compatibility}
    Let $\Phi \colon \cC \to \cD$ be a faithful triangulated functor between triangulated categories,
    and suppose $\tau$ and $\sigma$ are numerically compatible (along $\Phi$) stability conditions on $\cC$ and $\cD$, respectively. Then the following conditions are equivalent for the functor $\Phi$:
    \begin{enumerate}
        \item it respects (and reflects) stable objects;
        \item it respects (and reflects) Jordan--H\"older filtrations;
        \item it respects (and reflects) semistable objects, i.e. they are compatible as in Definition \ref{compatibility};
        \item it respects (and reflects) Harder--Narasimhan filtrations.
    \end{enumerate}
\end{lemma}
\begin{proof}
    Note that in any case, the reflect part is automatic, since if $F$ is not (semi)stable, any destabilizing sequence $F_\sub \to F \to F_\quot$ induces a destabilizing sequence $\Phi(F_\sub) \to \Phi(F) \to \Phi(F_\quot)$ of $\Phi(F)$. The numerically compatible assumption ensures they are nonzero and their slopes remain unchanged. 
    The equivalence of (i),(ii),(iii), and (iv) follows directly from the definition.
\end{proof}

Examples of numerically compatible stability conditions include Liu's construction of stability conditions on the tilted AP heart:
\begin{example}\label{extending is aligned}
    Any stability condition from the construction of Theorem \ref{stabilities on product} is numerically compatible with the original stability condition on $X$, along the natural pushforward $i_*\colon \dCat(X) \to \dCat(X \times C)$. 
    Indeed, for any $F \in \dCat(X)$, we have
    \begin{equation*}
        Z^{s,t,\beta}(i_*F)=b(i_* F) + \sqrt{-1} d (i_* F)=Z(p_* i_* F)=Z(F),
    \end{equation*}
    where $p \colon X \times C \to X$ is the projection.
\end{example}

The following theorem reveals that $\sigma$ and $\sigma^{s,t,\beta}$ are actually compatible.
\begin{theorem}\label{pushforward of stable obj}
    Let $X$ be a smooth projective variety, $C$ an elliptic curve, and $i \colon X \to X \times C$ the inclusion of any fiber of the natural projection $q \colon X \times C \to C$.
    Let $\sigma \in \Stab(X)$ and suppose that $F \in \dCat(X)$ is $\sigma$-stable.
    Then for any $\sigma^{s,t,\beta} \in \Stab(X \times C)$ constructed in Theorem~\ref{stabilities on product}, $i_*F$ is $\sigma^{s,t,\beta}$-stable.
\end{theorem}

The proof of this theorem is given at the end of this section. From this theorem, we immediately obtain the following corollary.

\begin{corollary}\label{pushforward of filtration}
    In the same situation as the previous theorem, if
    \begin{equation*}
        0 = F_0 \to  F_1 \to \dots \to  F_n = F
    \end{equation*}
    is the HN filtration of $F \in \dCat(X)$ with respect to some $\sigma \in \Stab(X)$, then
    \begin{equation*}
        0 = i_*F_0 \to i_* F_1 \to \dots \to  i_* F_n = i_* F
    \end{equation*}
    is the HN filtration of $i_* F \in \dCat(X \times C)$ with respect to any $\sigma^{s,t,\beta} \in \Stab(X \times C)$ constructed in Theorem~\ref{stabilities on product}.
\end{corollary}
\begin{proof}
    This directly follows from Lemma \ref{aligned compatibility}, Example~\ref{extending is aligned}, and Theorem \ref{pushforward of stable obj}.
\end{proof}


It is natural to ask whether this is true for general $C$. We will come back in Corollary~\ref{pushforward of stable obj g>1} to show that this holds true for any $C \not\simeq \PP^1$.




Before the proof of the theorem, we do the following preparation.
First, we state the following easy fact describing the subobjects in a tilted heart.

\begin{lemma}\label{sub in tilted heart}
    Let $\cA= \langle \cF,\cT \rangle$ be a heart of a triangulated category $\cD$ where $(\cT,\cF)$ is a torsion pair, and let $\cA^\sharp = \langle \cT, \cF[1]\rangle$ be the tilted heart. 
    For any $E \in \cT$,
    to give a subobject $E_{\sub} \hookrightarrow E$ in the tilted heart $\cA^\sharp$ is equivalent of giving a object $E_{\sub} \in \cT$ with a map $f\colon E_{\sub} \to E$ whose kernel (as a map in the abelian category $\cA$) satisfies $\ker f \in \cF$.
\end{lemma}
\begin{proof}
    The proof is straightforward, see, for example, \cite[Proposition 2.4]{Bay18}.
\end{proof}

Then, we observe the following fact which makes it possible to reduce the `scheme-theoretic support' of a destabilizing object with its numerics unchanged.

\begin{lemma}\label{reduced reduction}
    Let $X$ and $S$ be two smooth projective varieties.
    Let $i \colon X \to X \times S$ be the closed immersion of a fiber and $p \colon X \times S \to X$ the projection. Then for any $E \in \dCat_X(X\times S)$, one has $[i_* p_* E ]=[E]$ in $\K_{\QQ}(X\times S)$ (hence also in $\K_{\num} (X \times S)$).
\end{lemma}
\begin{proof}
    By the Grothendieck--Riemann--Roch formula, one has
    \begin{equation*}
        \ch ((i\circ p)_* E ) = (i\circ p)_*(\ch (E) \td (T_{i\circ p}))
    \end{equation*}
    Note that $\td(T_{i\circ p})=1$ and $(i\circ p)_*$ acts as the identity on any cycles supported on $X$, so $\ch(i_*p_*E) = \ch(E)$ in $\CH_\QQ$. Since $\ch \colon \K_\QQ \xrightarrow{\sim} \CH_\QQ$, one has $[i_* p_* E ]=[E]$ in $\K_\QQ$ as well.
\end{proof}

Now we can deal with the $C$-torsion case.
Recall that $E\in \dCat(X\times C)$ is called \emph{$C$-torsion} 
if its support is contained in $X \times Z$ for some proper closed subset $Z \subset C$.
\begin{proposition}\label{destabilizing reduction}
    Let $i \colon X \to X \times C $ be the closed immersion of a fiber of the projection $X \times C \to C$ where $C$ is a smooth projective curve. If $F$ is $\sigma$-stable for some $\sigma \in \Stab(X)$, then $i_* F$ admits no $C$-torsion destabilizing sequence for any stability condition constructed in Theorem~\ref{stabilities on product}.
\end{proposition}

\begin{proof}
    Up to a shift, we may assume $F \in \dCat(X)^{[0,0]} \eqdef \cA$.
    Then by definition we know $i_* F \in \cA^{t,\beta}_C$ since it is in $\cT^{t,\beta}$.
    Assume
    \begin{equation}\label{eq:destabilizing sequence}
        E_\sub \xrightarrow{f} i_* F \to E_\quot
    \end{equation}
    is a destabilizing sequence in $\cA^{t,\beta}_C$ with $E_\sub$ $C$-torsion. In particular, we have $c(E_\sub)=0$ where $c(-)$ is defined in \eqref{eq:def abcd}.
    By Lemma \ref{sub in tilted heart}, we know $E_\sub \in \cT^{t,\beta}$ and $\ker(f) \in \cF^{t,\beta}$.
    Thus by \eqref{eq:abcd}, we know that $c(\ker f)=0$. 
    If $\ker f \neq 0$, it will contradict with $\ker f \in \cF^{t,\beta}$. So it forces $\ker f =0$, i.e. $E_\sub$ is a subobject of $i_*F$ in $\cA_C$.
    Consequently, \eqref{eq:destabilizing sequence}
    is in fact a short exact sequence in $\cA_C$. 
    Note that every term is $C$-torsion, hence invariant under tensoring with $q^*L$.
    Therefore, applying $p_*$ where $p$ is the projection $p \colon X \times C \to X$, we have that
    \begin{equation}
        p_*E_\sub \to F \to p_*E_\quot
    \end{equation}
    is a short exact sequence in $\cA$, by the definition of AP heart. Applying $i_*$, one has
    \begin{equation}
        i_* p_*E_\sub \to i_* F \to i_* p_*E_\quot.
    \end{equation}
    Note that $E_\sub$ has to support on $X$, so by Lemma \ref{reduced reduction} we know $Z(i_* p_*E_{\sub})= Z( E_{\sub})$, hence $p_* E_{\sub}$ destabilizes $F$ in $\cA$, which is a contradiction with the stability of $E$.
\end{proof}

\begin{proof}[Proof of Theorem \ref{pushforward of stable obj}]
    We know that $i_*F$ is in the heart because $i_*F \in \cT^{t,\beta}$.
    Now consider the first HN factor $E_{\sub} \hookrightarrow i_* F$ and hence a sequence
    \begin{equation*}
        E_\sub \to i_* F \to E_\quot
    \end{equation*}
    By Proposition \ref{filtration reduction}, we have $E_{\sub}$ and hence also $E_{\quot}$ are supported on $X$, so we reduced to the case of Proposition \ref{destabilizing reduction}.
\end{proof}

\subsection{Scheme-theoretic support}\label{section:complex supp}
Let $E \in \dCat(X)$. We want a notion of scheme-theoretic support for $E$, in the sense that $E$ lies in the essential image of the pushforward from this support.

For any $L \in \Coh(X)$, consider the adjunction
\begin{equation*}
    \Hom(L \otimes E , L \otimes E) \simeq \Hom( L, \RcHom(E, L \otimes E )).
\end{equation*}
The identity on the left-hand side provides a canonical map
\begin{equation*}
    L \to \RcHom(E, L \otimes E).
\end{equation*}
Since $L$ is a sheaf, this gives a canonical map
\begin{equation}\label{eq:action map}
    \rho \colon \Gamma (L) \to \Hom(E, L \otimes E).
\end{equation}
by considering maps from $\cO_X$.
It is a standard fact that every subscheme $Y \subset X$ is defined by some sections of a sufficiently ample line bundle.
So it is reasonable to make the following definition.
\begin{definition}
    We define the \emph{scheme-theoretic support} of $E$ to be
    \begin{equation*}
    \Supp(E) \defeq V \left( \bigcup_{L\in \Pic(X)}\left\{ t \in \Gamma (L) \mid t \colon E \to E \otimes L \text{ is zero} \right\} \right).
    \end{equation*}
\end{definition}






This is a joint generalization of the scheme-theoretic support of a coherent sheaf and the set-theoretic support $\supp(E)$ of a complex $E$ in the derived category.

We will not use this definition in general but only focus on one special case of testing whether the scheme-theoretic support of $E$ lies in $X_0$ where $i \colon X_0 \to X_0 \times \PP^n \eqdef X$ is the natural closed immersion of a fiber for a point in $\PP^n$.
In this case, we have the natural projection
\begin{equation*}
    p \colon X_0 \times \PP^n \to X_0
\end{equation*}
which is the left inverse of $i$.
Besides, since $X_0 = V(t_1 ,\dots ,t_n)$ where $t_k$ is a global section of $\cO_X(H)$ and $H\defeq q^* \cO_{\PP^n}(1)$ with $q\colon X_0 \times \PP^n \to \PP^n$, the condition $\Supp(E) \subset X_0$ is equivalent to
\begin{equation}\label{eq:vanishing ideal}
    \rho(t_k)=0
\end{equation}
for every $k$.
We also note that once we know $E$ is (set-theoretically) supported on $X_0$, one can switch our setting to $X_0 \subset X_0 \times C^n$ interchangeably, where $C$ is any smooth curve. This is due to the equivalence
\begin{equation*}
    \dCat_{X_0}(X_0 \times \PP^n) \simeq \dCat_{X_0}(X_0 \times C^n) ;
\end{equation*}
see, for example, \cite[Corollary 2.9]{Orl11}.



The following lemma justifies the definition in our special case.

\begin{lemma}\label{supp by vanishing ideal}
    In the above situation, suppose $\Supp(E) \subset X_0$, i.e. \eqref{eq:vanishing ideal} holds for every $k$. Then there exists some $F \in \dCat(X_0)$ such that $E \simeq i_* F$
\end{lemma}

\begin{proof}
    It is easy to check that $E$ is (set-theoretically) supported on $X_0$, so we may assume $i \colon X_0 \to X_0 \times C^n$, where $C$ is any smooth projective curve.
    Write $i$ as $i_1 \circ i_2 \circ \dots \circ i_n$, where $i_k \colon X_0 \times C^{k-1} \to X_0 \times C^k$ is the natural closed immersion.
    Note that we have the triangle
    \begin{equation*}
    	E \xrightarrow{t_k} E \to i_{k*} i_k^* E
    \end{equation*}
    for $E \in \dCat(X_0 \times C^k)$, so the condition \eqref{eq:vanishing ideal} implies that
    \begin{equation*}
        i_{k*}i_k^*E \simeq E \oplus E[1].
    \end{equation*}
    for every $k$.
    Apply this iteratively, one obtains an isomorphism
    \begin{equation}\label{eq:dec}
        i_*i^*E \simeq \bigoplus_k \binom{n}{k} E[k].
    \end{equation}
    Apply $i_*p_*$ to it, we have
    \begin{equation}\label{eq:dec'}
        i_*i^*E \simeq i_*p_*i_*i^*E \simeq \bigoplus_k \binom{n}{k} i_*p_*E[k].
    \end{equation}
    By Remark~\ref{Krull--Schmidt}, the bounded derived category of coherent sheaves on a smooth projective variety is Krull--Schmidt. 
    Therefor, any object has a unique way, up to isomorphism, to write as direct sum of indecomposable objects. As a consequence, by comparing \eqref{eq:dec} with \eqref{eq:dec'}, we have
    \begin{equation*}
        E \simeq i_*p_* E ,
    \end{equation*}
    so we can conclude by setting $F=p_*E$.
\end{proof}

    \begin{remark}
        See another more elementary approach in \cite[Section~2.4]{AP06}.
    \end{remark}


The following lemma relates the `vanishing criterion' with the `extendable criterion' (see the remark below).

\begin{lemma}\label{supp by extendability}
    Let $E\in \dCat(X_0 \times \PP^n)$, and let $\delta_l ~ (l = 0 , \dots , n)$ be the morphism defined by
    \begin{align*}
        \delta_l \colon  X_0 \times \PP^n &\to X_0 \times \PP^{n+1} \\
        (x, a_1,\dots, a_n) &\mapsto \begin{cases} (x, a_1,\dots, a_n, 0),  &\text{if $l=0$};\\ (x, a_1,\dots, a_n, a_l), &\text{if $l \neq 0$}.\end{cases}
    \end{align*}
    Assume that there is an isomorphism
    \begin{equation}\label{eq:extendable on degree 1}
        \varphi \colon \delta_{0 *}E \xrightarrow{\sim} \delta_{l *} E
    \end{equation}
    for any $l$. Then there exists $F\in \dCat(X_0)$ such that $E \simeq i_* F$.
\end{lemma}

\begin{proof}
    Due to Lemma~\ref{supp by vanishing ideal}, it suffices to check \eqref{eq:vanishing ideal} holds for every $k=1,\dots, n$.
    From the construction, the action map \eqref{eq:action map} for $\delta_{k*} E $, say $\rho_k$, factors through the action map for $E$, i.e.
    \begin{equation*}
        \begin{tikzcd}
        \Gamma(\cO_{X_0\times\PP^{n+1}}(H)) \arrow[d] \arrow[r,"\rho_k"] & {\Hom(\delta_{k*}E,\delta_{k*}E(H))}           \\
        \Gamma(\cO_{X_0 \times \PP^n}(H)) \arrow[r,"\rho"]                        & {\Hom(E,E(H))} \arrow[u, hook]
        \end{tikzcd}
    \end{equation*}
    where the map $\Hom(E,E(H)) \hookrightarrow \Hom(\delta_{k*}E,\delta_{k*}E(H))$ is injective since it is an inclusion of a direct summand of 
    \begin{equation*}
    	\Hom(\delta_{k*} E,\delta_{k*}E(H)) \simeq \Hom(\delta_k^* \delta_{k*}E,E(H)) \simeq \Hom(E \oplus E[1],E(H)) . 
    \end{equation*}
    Note that the map $\Gamma(\cO_{X \times \PP^1}(H)) \to \Gamma(\cO_X(H))$ is given by 
    \begin{align*}
    	\Gamma(\cO_{X_0 \times \PP^{n+1}}(H))\simeq \CC[t_1, \dots ,t_{n+1} ] &\to \CC [t_1, \dots, t_n] \simeq \Gamma(\cO_{X_0 \times \PP^n}(H)) \\
	t_{n+1} &\mapsto \begin{cases} 0, &\text{if $k = 0$}; \\ t_k, &\text{if $k \neq 0$}. \end{cases}
    \end{align*}    
    From the construction, via $\delta_0$ (or any $\delta_l$ with $l \neq k$) one has
    \begin{equation*}
        \rho_0(t_{n+1})=\rho(0)=0\in \Hom(E,E(H))\subset \Hom(\delta_{0*}E,\delta_{0*}E(H))
    \end{equation*}
    while via $\delta_k$ one has
    \begin{equation*}
        \rho_k(t_{n+1})=\rho(t_k)\in \Hom(E,E(H))\subset \Hom(\delta_{k*}E,\delta_{k*}E(H))
    \end{equation*}    
    From the assumption, it forces $\rho(t_k)=0$ for any $k=1,\dots , n$, so we are reduced to Lemma~\ref{supp by vanishing ideal}.
\end{proof}

\begin{remark}
    Following \cite{LO22}, one may multiply by $C$ repeatedly and define a cosimplicial scheme $X^\bullet \defeq X_0 \times C^ \bullet$, and its derived category of cosimplicial sheaves is equivalent to the derived category of sheaves from $X_0$, see \cite[Theorem~2.6 and Example~2.13]{LO22}.
    Also note that the cosimplicial condition at higher degrees is generated by the condition at degree one in a way of \cite[Section~2.16]{LO22}, and the latter is precisely \eqref{eq:extendable on degree 1}.
    However, this only defines an object in the cosimplicial derived category of sheaves, instead of the derived category of cosimplicial sheaves. They are related if certain negative Ext groups vanish (cf. \cite[Theorem~1.5 and Theorem~1.7]{Olsson:2024}), which is not in general true in our cases.
    It would be interesting to know how to approach in this way.
\end{remark}

And we finally achieved the `rotation-invariant criterion'.
\begin{lemma}\label{supp by invariance}
    For any $n \geq 1$, let $i \colon X_0 \to X_0 \times \PP^n$ and $\delta_0 \colon X_0 \times \PP^n \to X_0 \times \PP^{n+1}$ be the natural closed immersions as above. 
    Suppose $E \in \dCat(X_0 \times \PP^n)$ such that $\delta_{0*}E \in \dCat_{X_0}(X_0 \times \PP^{n+1})$ is invariant under the natural $\GL_{n+1}$-action (i.e. automorphisms of $\PP^{n+1}$ that fixes the point where the image of $\delta_0 \circ i$ stands).
    Then there exists some $F \in \dCat(X_0)$ such that $E \simeq i_* F$.
\end{lemma}
\begin{proof}
    By assumption, for any $g \in \GL_{n+1}$ there is an isomorphism
    \begin{equation*}
        \varphi_g \colon \delta_{0*}E \xrightarrow{\sim} g_*\delta_{0*}E.
    \end{equation*}
    Take $g$ to be the shear transform
    \begin{equation*}
        \begin{pmatrix}
            1 &   &   & \\
              & \ddots &  & 1\\
              &   &  1 &\\
              &   &   & 1 \\
        \end{pmatrix}
    \end{equation*}
    which is the identity matrix plus the elementary matrix with $1$ at the $(l,n+1)$-entry,
    then $g_*\delta_{0*} = \delta_{l*}$. So we are reduced to Lemma~\ref{supp by extendability}.
\end{proof}

\subsection{Scheme-theoretic restriction}\label{section:restrict scheme}
We have discussed when an object lies in the essential image of the pushforward along a closed immersion, and now we want to study the morphisms as well.
Nevertheless, the following proposition suggests that, if we only concern the morphisms in an HN filtration, we only need to control the objects.
\begin{lemma}\label{supp of morphisms}
    Let $X$ be a smooth projective variety and $Y$ be a smooth subvariety of $X$ such that the closed immersion $i \colon Y \hookrightarrow X$ admits a retraction.
    \begin{enumerate}
        \item For any $F_1 , F_2 \in \dCat(Y)$  the natural map
        \begin{equation}\label{eq:morphism}
            \Hom(F_1,F_2) \to \Hom(i_*F_1,i_* F_2)
        \end{equation}
        is injective.
        \item If moreover $F_1$ and $F_2$ are in the heart of some bounded $t$-structure on $\dCat(Y)$.
    Then, the natural map \eqref{eq:morphism} is an isomorphism.
    \end{enumerate}
\end{lemma}

\begin{proof}
    Since $i$ admits a retraction, $i_*$ is faithful on the level of objects, so
    \begin{equation*}
        i^*i_* F \simeq \bigoplus_{k=0}^n \wedge^k \cN_{Y/X}[k]\otimes F \simeq \bigoplus_{k=0}^n F^{\oplus \binom{n}{k}}[k] ,
    \end{equation*}
    where $n$ is the codimension of $Y$ in $X$; see, for example, \cite[Proposition~11.1]{Huybrechts:2006}.
    Note that the counit at $F$, $\epsilon_F \in \Hom(i^*i_*F,F)$, is just the natural projection
    \begin{equation*}
        \begin{tikzcd} [row sep = tiny]
        	{i^*i_*F \simeq  F \oplus \left( \bigoplus_{k=1}^n F^{\oplus \binom{n}{k}}[k] \right)} & F , \\
        	{(e, \dots)} & e .
        	\arrow[from=1-1, to=1-2]
        	\arrow[maps to, from=2-1, to=2-2]
        \end{tikzcd}
    \end{equation*}
    This provides a (split) distinguished triangle
    \begin{equation*}
        \bigoplus_{k=1}^{n} F^{\oplus \binom{n}{k}} [k] \to i^*i_* F \xrightarrow{\epsilon_F} F .
    \end{equation*}
    Let $F=F_1$, then the morphisms from this triangle to $F_2$ yield a long exact sequence
    \begin{equation*}
        0 \to \Hom(F_1,F_2) \to \Hom(i^* i_* F_1,F_2) \to  \Hom(\bigoplus_{k=1}^{n} F_1^{\oplus \binom{n}{k}} [k],F_2) \to \dots
    \end{equation*}
    where the morphism $\Hom(F_1,F_2) \to \Hom(i^* i_* F_1,F_2)$ is equivalent to \eqref{eq:morphism} by adjunction. If $F_1$ and $F_2$ are in the same heart, then the last term vanishes as well.
\end{proof}


Now we state the main theorem of this section.

\begin{theorem} \label{thm:main}
    Let $f \colon X \to Y$ be a morphism between smooth projective varieties such that the Albanese morphism of $Y$ is finite. 
    Let $X_0$ be a fiber of $f$ such that $f$ is isotrivial near the fiber $X_0$, and let $i\colon X_0 \to X$ be the closed immersion.
    Then, there is a natural restriction 
    \begin{equation*}
        \Stab_{\Lambda}(X) \to \Stab_{\Lambda_0}(X_0), \qquad (\cP ,Z) \mapsto (\cP_0, Z_0),
    \end{equation*}
    in the way of Lemma~\ref{formal restriction} along $i_*$. Namely,
    \begin{equation*}
        \cP_0(\phi) = (i_*)^{-1} \cP(\phi) = \{F \in \dCat(X_0) \mid i_* F \in \cP(\phi)\}, \qquad Z_0 =Z \circ [i_*] ,
    \end{equation*}
    and it satisfies the support property with respect to $\K( X_0 ) \xdhrightarrow{} \Lambda_0,$
    where $\Lambda_0$ is the image of the composition of maps $\K(X_0) \xrightarrow{[i_*]} \K(X) \xdhrightarrow{v} \Lambda$.
    In particular, the restriction respects geometric stability conditions.
\end{theorem}

By Corollary~\ref{pre-restriction'}, it remains to restrict a stability condition on $\dCat_{X_0}(X)$ to $\dCat(X_0)$, which is the following proposition.

\begin{proposition}\label{restriction}
    Let $\sigma'$ be a stability condition on $\dCat(X)$ where $X_0=q^{-1}(0)$ is some smooth fiber of a morphism $q\colon X \to S$, and let $i \colon X_0 \to X$ be the closed immersion.
    Assume that $q$ is isotrivial near $X_0$, and for any $F \in 
    \dCat(X_0)$, all the HN factors of $i_*F$ with respect to $\sigma'$ are set-theoretically supported on $X_0$.
    Then, the stability condition $\sigma'$ restricts to a stability condition $\sigma$ on $\dCat(X_0)$, in the way of Lemma~\ref{formal restriction} along $i_*$.
    In particular, it respects geometric stability conditions.
\end{proposition}
Here, by isotrivial we mean that there exists a finite (\'{e}tale) morphism such that the base-change of the morphism is a projection.
This assumption allows us to reduce it to the trivial bundle case.
\begin{lemma}\label{trivialization}
    If $q\colon X \to S$ is isotrivial and $X_0=q^{-1}(0)$ is some fiber, then $\dCat_{X_0}(X) \simeq \dCat_{X_0}(X_0 \times C^n)$, where $C$ is any smooth projective curve and $n=\dim(S)$, and $X_0\subset X_0\times C^n$ is a fiber of the natural projection $X_0 \times C^n \to C^n$.
\end{lemma}
\begin{proof}
    The isotrivial assumption implies that there exists a finite \'{e}tale morphism $f \colon S' \to S$  such that the base change $X' \simeq X_0 \times S'$ as $S'$-schemes.
    Let $t \in f^{-1}(0)$, then 
    \begin{equation*}
        \widehat{X}_{X_0} \simeq \widehat{X'}_{X'_t} \simeq \widehat{(X_0 \times S)}_{X_0} ,
    \end{equation*}
    hence
    \begin{equation*}
        \dCat_{X_0}(X) \simeq \dCat_{X_0}(X_0\times C^n)
    \end{equation*}
    by \cite[Corollary 2.9]{Orl11}.
\end{proof}

Now we can prove Theorem~\ref{thm:main} and Proposition~\ref{restriction}.
\begin{proof}[Proof of Proposition~\ref{restriction}, hence Theorem~\ref{thm:main}]
    By Lemma~\ref{trivialization}, one may assume $X = X_0 \times C^n$ where $C$ is any smooth projective curve.
    
    Let $\sigma'=(\cP',Z')$ be the given stability condition on $X_0 \times C^n$. One can easily define the candidate
    \begin{equation*}
        \cP(\phi)=\{E \in \dCat(X_0) \mid i_*(F) \in \cP'(\phi)\} \quad \text{and} \quad
        Z=Z'\circ i_* 
    \end{equation*}
    in the manner of Lemma~\ref{formal restriction},
    To show this is a stability condition, it suffices to show that the HN filtration exists. For any nonzero $F \in \dCat(X_0)$, $i_* F$ is nonzero, so there exists an HN filtration in $\dCat(X_0\times C^n)$
    \begin{equation}\label{eq:filtration}
        0 =E_0 \to  E_1 \to \dots \to  E_n = i_*F
    \end{equation}
    with respect to this given $\sigma'$.
    Once we proved this filtration is pushforward from a filtration on $\dCat(X_0)$, we can conclude by Lemma~\ref{formal restriction};
    in fact we have $\K(X_0) \simeq \K(\dCat_{X_0}(X))$, so we can actually use the same lattice for the support property.

    To this end, we first embed it into a family with base $1$-dimensional higher. Let $C$ be an elliptic curve and $\delta_0\colon X \to X\times C$ be the closed immersion of any fiber of the projection $X \times C \to C$. Then by Corollary \ref{pushforward of filtration} there exists some $\sigma''\in \Stab(X\times C)$ such that 
    \begin{equation*}
        0 = \delta_{0*} E_0 \to \delta_{0*}  E_1 \to \dots \to  \delta_{0*} E_n = \delta_{0*}i_*F
    \end{equation*}
    is the HN filtration with respect to $\sigma''$.
    By Corollary~\ref{pre-restriction}, one can view this as on $\dCat_{X_0}(X \times C)$ and by Lemma~\ref{trivialization} we know 
    $\dCat_{X_0}(X \times C) \simeq \dCat_{X_0}({X_0} \times \PP^{n+1})$.
    Then we have the $\GL_{n+1}$-action on $\dCat_{X_0}({X_0} \times \PP^{n+1})$ and hence on $\Stab({X_0})_{{X_0} \times \PP^{n+1}}$, which has to be trivial by Theorem~\ref{invariant action}. Finally, by Lemma~\ref{supp by invariance} the invariance condition ensures $\delta_{0*} E_k \simeq (\delta_{0} \circ i)_* F_k$ and hence $E_k \simeq i_* F_k$ for some $F_k \in \dCat(X_0)$ for every $k$, and by Lemma~\ref{supp of morphisms} we know that there exists (a unique) morphism $f_k \colon F_k \to F_{k+1}$ such that
    \begin{equation}\label{eq:compatible with pushforward}
        \begin{tikzcd} 
        E_k \arrow[r] \arrow[d, phantom, "\vsimeq"] & E_{k+1} \arrow[d, phantom, "\vsimeq"] \\
        i_* F_k \arrow[r,"i_*f_k"]                     & i_* F_{k+1}                    
        \end{tikzcd}
    \end{equation}
    for every $k$.
    These $F_k$ and $f_k$ build the filtration
    \begin{equation*}
        0 =F_0 \xrightarrow{f_0}  F_1 \xrightarrow{f_1} \dots \to  F_n = F,
    \end{equation*}
    and applying the pushforward $i_*$, it gives \eqref{eq:filtration} by \eqref{eq:compatible with pushforward}, as desired.
\end{proof}

\begin{remark}
    In fact, by Remark~\ref{local stability conditions} we can prove the following:
    Let $\sigma'$ be a stability condition on $\dCat_{X_0}(X)$ where $X_0=q^{-1}(0)$ is a some smooth fiber of $q\colon X \to S$. If moreover $q$ is isotrivial near $X_0$, the stability condition $\sigma'$ restricts to a stability condition $\sigma$ of $\dCat(X_0)$ in the way of Lemma~\ref{formal restriction} along $i_*$.
\end{remark}

A direct corollary is the following, which can be viewed as some converse of Liu's Theorem~\ref{stabilities on product}.
\begin{corollary}
    Assume $C \not\simeq \PP^1$ is any smooth projective curve. For any stability condition $\sigma^{s,t,\beta}$ on $X \times C$ induced from a stability condition $\sigma$ on $X$ via Theorem~\ref{stabilities on product}, we can restrict it back to $X$ in the way of Lemma~\ref{formal restriction} along $i_*$, where $i \colon X \to X \times C$. Moreover, the restriction gives the original $\sigma$, i.e. $\sigma^{s,t,\beta}|_{X} = \sigma$.
\end{corollary}

\begin{proof}
    Since $g(C) \geq 1$, $\alb \colon C \to \Alb(C)$ is a closed immersion, hence $X$ is a fiber of the composition 
    \begin{equation*}
        X \times C \longrightarrow C \xrightarrow{\alb_C} \Alb(C) .
    \end{equation*}
    By Theorem~\ref{thm:main}, we can define such restriction.
    By construction, the pushforward $i_*$ of the HN filtration of any $F \in \dCat(X)$ with respect to $\sigma^{s,t,\beta}|_{X}$ is precisely the HN filtration of $i_*F$ with respect to $\sigma^{s,t,\beta}$, hence $\sigma^{s,t,\beta}|_{X} = \sigma$.
\end{proof}

This also strengthens the result of Theorem~\ref{pushforward of stable obj}.
\begin{corollary}\label{pushforward of stable obj g>1}
    We adopt the setting of Theorem~\ref{pushforward of stable obj} except that we only require $g(C) \geq 1$. Then $i_*F$ is $\sigma^{s,t,\beta}$-stable.
\end{corollary}


\section{Applications: Construct stability conditions on some HK or CY manifolds}\label{sec:example}

In this section, we provide examples of constructions of stability conditions on some hyper-K\"{a}hler or strict Calabi--Yau manifolds, by using all the techniques we developed in previous sections.


\subsection{Stability conditions on certain hyper-K\"{a}hler manifolds of generalized Kummer type}
Let $C_i~(i=1,2)$ be two elliptic curves, and let $S$ be an abelian surface which is isogenous to $C_1 \times C_2$.
Consider the generalized Kummer variety $K_n(S)$ associated to the abelian surface $S$, i.e.,
\begin{equation*} 
    K_n(S)= \alpha_n^{-1}(0)  ,
\end{equation*}
where $\alpha_n \colon  S^{[n]} \to S$ is the Albanese map for the Hilbert scheme $S^{[n]}$.

We now prove that there exist stability conditions on $K_n(S)$.
\begin{proof}[Proof of Theorem~\ref{thm:Kummer}]
    We may assume $S = (C_1 \times C_2) /G$, where $G$ is a finite group acting freely on $C_1 \times C_2$.
    As in Example~\ref{BKR free}, we have an equivalence
    \begin{equation} \label{eq:BKR Kummer}
        \dCat(C_1 \times C_2)^G \simeq \dCat(S).
    \end{equation}
    Let $G^n \rtimes \fS_n \subset \Aut((C_1 \times C_2)^n)$ be the natural group acting on $(C_1 \times C_2)^n$.
    Then we have
    \begin{align*}
        \dCat \left((C_1 \times C_2)^n \right)^{G^n \rtimes \fS_n}
        &\simeq \left(\dCat \left((C_1 \times C_2)^n \right)^G \right)^{\fS_n} \quad \text{(Lemma \ref{successive invariants})} \\
        &\simeq \dCat( S^n)^{\fS_n}  \quad \text{(Lemma \ref{exterior tensor product} and 
        \eqref{eq:BKR Kummer})} \\
        &\simeq \dCat( S^{[n]})  \quad \text{(Theorem~\ref{BKR Haiman}).}
    \end{align*}
    By Proposition~\ref{special stability via curve product} and Theorem~\ref{pre-restriction to end}, there exist $(G^n \rtimes \fS_n)$-invariant stability conditions on $(C_1 \times C_2)^n$.
    And by Theorem~\ref{inducing stability conditions}, they induce stability conditions on $\dCat \left((C_1 \times C_2)^n \right)^{G^n \rtimes \fS_n} \simeq \dCat( S^{[n]})$.
    Note that $K_n(S)$ is a fiber of $\alpha_n \colon S^{[n]} \to S$ which is isotrivial since there is a fiber product diagram
    \begin{equation*}
        \begin{tikzcd} [column sep = small]
        	{K_n(S) \times S} & {S^{[n]}} \\
        	S & S
        	\arrow[from=1-1, to=1-2]
        	\arrow[from=1-1, to=2-1]
        	\arrow["{\alpha_n}", from=1-2, to=2-2]
        	\arrow["n"', from=2-1, to=2-2]
        \end{tikzcd}
    \end{equation*}
    where the isomorphism is given by
    \begin{align*}
        S^{[n]}  \times_S S &\xrightarrow{\sim} K_n(S) \times S , \\
        (Z,a) &\mapsto (Z-a,a).
    \end{align*}
    By Theorem~\ref{thm:main}, we obtain stability conditions on $K_n(S)$.
\end{proof}

\subsection{Stability conditions on certain Calabi--Yau manifolds of even dimension}
Let $C_i~(i=1,2)$ be two elliptic curves.
Consider the actions
\begin{equation}\label{eq:actions}
    \begin{aligned}
            -\id \colon C_1 \times C_2 &\to C_1 \times C_2 \\
            (x,y) &\mapsto (-x,-y) \\
            \text{and} \quad
            \tau \colon C_1 \times C_2 &\to C_1 \times C_2 \\
            (x,y) &\mapsto (-x+t_1,y+t_2)
    \end{aligned}
\end{equation}
where $t_i$ is any fixed 2-torsion point on $C_i$.
We denote the resolution of $(C_1 \times C_1) / \langle -\id \rangle$ by $S'$, which is well-known as a Kummer (K3) surface,
and the resolution of $(C_1 \times C_2) / \langle -\id ,\tau \rangle$ by $S$.
One easily checks that $\tau$ descends to $(C_1 \times C_1) / \langle -\id \rangle$ and lifts to a free involution on $S'$ which we will still denote by $\tau$. Therefore, $S = S' / \langle \tau \rangle$ is an Enriques surface.

Take the Hilbert scheme of $n$ points for $n \geq 2$. According to \cite[Theorem 3.1]{OS11}, the universal cover $f \colon \widetilde{S^{[n]}} \to {S^{[n]}}$ is $2:1$ and $\widetilde{S^{[n]}}$ is a strict Calabi--Yau manifold.

Now we prove that there exist stability conditions on this $\widetilde{S^{[n]}}$, thus provide examples of stability conditions on certain strict Calabi--Yau manifolds in arbitrary even dimension.

\begin{proof}[Proof of Theorem~\ref{thm:CY even}]
    By Lemma~\ref{successive invariants} and Theorem~\ref{BKR}, we have equivalences
    \begin{equation} \label{eq:BkR enriques}
        \dCat(C_1 \times C_2)^{\ZZ/2 \times \ZZ/2} \simeq \dCat(S')^{\langle \tau \rangle} \simeq \dCat(S)
    \end{equation}
    and by construction they are Fourier--Mukai transforms.
    Here $\ZZ/2 \times \ZZ/2 \simeq \langle -\id, \tau \rangle$ acts on $C_1 \times C_2$ as described in \eqref{eq:actions}.
    Let $(\ZZ/2 \times \ZZ/2)^n \rtimes \fS_n \subset \Aut \left( (C_1 \times C_2)^n \right)$ be the natural group acting on $(C_1 \times C_2)^n$.
    Then we have
    \begin{align*}
        \dCat \left((C_1 \times C_2)^n \right)^{(\ZZ/2 \times \ZZ/2)^n \rtimes \fS_n}
        &\simeq \left(\dCat \left((C_1 \times C_2)^n \right)^{(\ZZ/2 \times \ZZ/2)^n}\right)^{\fS_n} \quad \text{(Lemma \ref{successive invariants})} \\
        &\simeq \dCat( S^n)^{\fS_n}  \quad \text{(Lemma \ref{exterior tensor product} and \eqref{eq:BkR enriques})} \\
        &\simeq \dCat( S^{[n]})  \quad \text{(Theorem~\ref{BKR Haiman}).}
    \end{align*}
    By Proposition~\ref{special stability via curve product} and Theorem~\ref{pre-restriction to end}, there exist $((\ZZ/2 \times \ZZ/2)^n \rtimes \fS_n)$-invariant stability conditions on $(C_1 \times C_2)^n$; and by Theorem \ref{inducing stability conditions}, they induce stability conditions on $\dCat((C_1 \times C_2)^n)^{(\ZZ/2 \times \ZZ/2)^n \rtimes \fS_n} \simeq \dCat(S^{[n]})$.

    Now we want to lift these stability conditions to $\widetilde{S^{[n]}}$. 
    By Proposition \ref{inducing stability conditions abelian}, it is equivalent to check whether these stability conditions are invariant by the dual action of the covering involution of
    $f \colon \widetilde{S^{[n]}}\to S^{[n]}$,
    which is just tensoring with $\omega_{S^{[n]}}$, as in Example \ref{serre invariance}.
    We claim it is indeed true. As in Example~\ref{serre invariance} (or by Corollary~\ref{co-invariance description abelian} in general) we know $f^*(\omega_{S^{[n]}})\simeq \cO_{\widetilde{S^{[n]}}}$.
    Since $f^*$ is injective on cohomology, we know $\ch(\omega_{S^{[n]}})=1 \in \coho^*(S^{[n]},\QQ)$. Note that the equivalence $\dCat(S^{[n]}) \simeq \dCat \left((C_1 \times C_2)^n \right)^{(\ZZ/2 \times \ZZ/2)^n \rtimes \fS_n} \eqdef \dCat(\cX)^{\cG}$ is given by a Fourier--Mukai transform as in \eqref{eq:diagram BKR}. We borrow the notation there and write the morphisms as
    \begin{equation}\label{eq:diagram equivalence}
        \begin{tikzcd}
                             & \cZ \arrow[ld, "p"'] \arrow[rd, "q"]                      &                               \\
        {S^{[n]}} \arrow[rd] &                                                           & \cX \arrow[ld] \\
                             & \cX/{\cG} .&                              
        \end{tikzcd}
    \end{equation}
    Passing to cohomology, note that the higher degree terms of $\td(q)$ will be killed by $q_*$ since
    \begin{equation*}
        1=\ch(\cO_\cX)=\ch(q_* \cO_\cZ) = q_* (\ch(\cO_\cZ) . \td(q))=q_* \td(q),
    \end{equation*}
    hence
    \begin{equation*}
        \ch(q_*p^* \omega_{S^{[n]}})=q_*\left(\ch(p^* \omega_{S^{[n]}}).\td(q)\right)= q_*(\td(q))=1.
    \end{equation*}
    Here we use the same notation for functors on cohomology.
    This checks the central charges are invariant under the autoequivalence induced by tensoring with $\omega_{S^{[n]}}$. And by Theorem~\ref{pre-restriction to end}, we know that the stability conditions are indeed invariant under tensoring with $\omega_{S^{[n]}}$, as claimed.
    Therefore, we can lift these stability conditions to $\widetilde{S^{[n]}}$.
\end{proof}

\subsection{Stability conditions on certain Calabi--Yau manifolds of odd dimension}

Let $S$ be a bielliptic surface, i.e. $S \simeq (C_1 \times C_2)/G$ where $C_i~(i=1,2)$ are elliptic curves and $G$ is a finite group acting by translation on $C_1$ and by automorphism (of elliptic curves) on $C_2$. A classification of such surfaces has been worked out by Bagnera--de Franchis, see \cite[List VI.20]{Bea96}.

Now take the Hilbert scheme of $n$ points for $n \geq 2$. We know that the fundamental group of  $S^{[n]}$ is a finitely generated abelian group of rank two, so any of its (\'{e}tale) coverings is quotient by a free action of an abelian group.
Among these coverings, we can find a finite covering $f \colon \widetilde{S^{[n]}} \to S^{[n]}$ such that $\widetilde{S^{[n]}}\simeq X \times C_2$ and $X$ is a strict Calabi--Yau manifold, thanks to \cite[Theorem~3.5]{OS11}.

We now prove that there exist stability conditions on this $X$, thus provide examples of stability conditions on certain strict Calabi--Yau manifolds in arbitrary odd dimension.
\begin{proof}[Proof of Theorem~\ref{thm:CY odd}]
    Note that the action of $G$ on $C_1 \times C_2$ is always free, so as in Example~\ref{BKR free} we have a Fourier--Mukai equivalence
    \begin{equation}\label{eq:BKR elliptic}
        \dCat(C_1 \times C_2)^G \simeq \dCat(S).
    \end{equation}
    Let $G^n \rtimes \fS_n \subset \Aut ( (C_1 \times C_2)^n )$ be the natural group acting on $(C_1 \times C_2)^n$.
    Then we have
    \begin{align*}
        \dCat \left((C_1 \times C_2)^n \right)^{G^n \rtimes \fS_n}
        &\simeq \left(\dCat \left((C_1 \times C_2)^n \right)^{G^n} \right)^{\fS_n} \quad \text{(Lemma \ref{successive invariants})} \\
        &\simeq \dCat( S^n)^{\fS_n}  \quad \text{(Lemma \ref{exterior tensor product} and 
        \eqref{eq:BKR elliptic})} \\
        &\simeq \dCat( S^{[n]})  \quad \text{(Theorem~\ref{BKR Haiman}).}
    \end{align*}
    By Proposition~\ref{special stability via curve product} and Theorem~\ref{pre-restriction to end}, there exist $(G^n \rtimes \fS_n)$-invariant stability conditions on $(C_1 \times C_2)^n$; and by Theorem \ref{inducing stability conditions}, they induce stability conditions on $\dCat((C_1 \times C_2)^n)^{G \rtimes \fS_n} \simeq \dCat(S^{[n]})$

    Now we want to lift these stability conditions to $\widetilde{S^{[n]}}$. 
    By Proposition \ref{inducing stability conditions abelian}, it is equivalent to check whether these stability conditions are invariant by the dual action of the deck transformations of
    $f \colon \widetilde{S^{[n]}}\to S^{[n]}$,
    which by Corollary~\ref{co-invariance description abelian} is just tensoring with certain line bundles $L$ satisfying $f^*L \simeq \cO_{\widetilde{S^{[n]}}}$, hence $\ch(L)=1 \in \coho^*(S^{[n]},\QQ)$. 
    Now we are precisely in the same situation as in proof of Theorem~\ref{thm:CY even}.
    The equivalence $\dCat(S^{[n]}) \simeq \dCat \left((C_1 \times C_2)^n \right)^{G \rtimes \fS_n} \eqdef \dCat(\cX)^{\cG}$ is given by a Fourier--Mukai transform as in \eqref{eq:diagram equivalence},
    and the same computation there suggests $\ch(q_* p^* L)=1$.
    Therefore, the central charges are invariant under the autoequivalence induced by tensoring with $L$, and by Theorem~\ref{pre-restriction to end} we know that the stability conditions themselves are invariant as well.
    So we obtain stability conditions on $\widetilde{S^{[n]}}\simeq X \times C_2$.

    Finally, by Theorem~\ref{thm:main} again, we obtain stability conditions on $X$.
\end{proof}

\subsection{Stability conditions on Cynk--Hulek Calabi--Yau manifolds arising from $\ZZ/2$-actions}\label{section:CH}


Let $C_1, \dots , C_n$ be elliptic curves, and 
\begin{equation*}
    G_n \defeq \{ (a_1, \dots, a_n) \mid \sum a_i =0 \} \subset (\ZZ/2)^n
\end{equation*}
act on $C_1 \times \dots \times C_n$ via factor-wise involution.
A Cynk--Hulek Calabi--Yau manifold arising from $\ZZ/2$-actions is constructed as a crepant resolution of $(C_1 \times \dots \times C_n) /{G_n}$.
Explicitly, for any $n \geq 2$ we define
\begin{equation*}
    X_n \defeq \Hilb^{G_n^+}_\can(X_{n-1} \times C_n).
\end{equation*}
where the $\ZZ/2 \simeq G_n^+$-action on $X_{n-1} \times C_n$ is described in the following.
The natural action $(\id_{X_{n-2}} , -\id)$ on $X_{n-2} \times C_{n-1}$ gives rise to an action on $X_{n-1} = \Hilb_\can^{G_{n-1}^+}(X_{n-2} \times C_{n-1})$ which we shall denote by $(\id_{X_{n-2}} , -\id)^{[2]}$.
Then the $\ZZ/2 \simeq G_n^+$-action on $X_{n-1} \times C_n$ is generated by $\left(  (\id_{X_{n-2}} , -\id)^{[2]} , -\id \right)$.
To distinguish between these two actions, we will call the action generated by $(\id_{X_{n-2}} , -\id)^{[2]}$ the \emph{odd action} and denote it by $G_{n-1}^-$, and the action generated by $\left(  (\id_{X_{n-2}} , -\id)^{[2]} , -\id \right)$ the \emph{even action} and denote it by $G_n^+$.
We denote the quotient map by
\begin{equation*}
    \tau_n \colon X_n = \Hilb^{G_n^+}_\can(X_{n-1} \times C_n) \to (X_{n-1} \times C_n) / {G_n^+} \eqdef \overline{X}_n .
\end{equation*}
We adopt the convention
\begin{equation*}
    X_0= \emptyset, \quad X_1 = C_1, \quad \text{and} \quad  G_1^+=0
\end{equation*}
for $n=1$.

\begin{proposition}\label{equivalence Z/2}
    There is a Fourier--Mukai equivalence
    \begin{equation*}
        \dCat(C_1 \times \dots \times C_n) ^{G_n} \simeq \dCat(X_n)
    \end{equation*}
    for any $n\geq 1$.
\end{proposition}
\begin{proof}[Proof of Proposition~\ref{equivalence Z/2}]
    We claim that
    \begin{equation*}
        \text{$\tau_n$ is a crepant resolution} \quad \text{and} \quad \dCat(X_{n-1} \times C_n)^{G_n^+} \simeq \dCat(X_n) .
    \end{equation*}
    Once we know this, we have 
    \begin{equation}\label{eq:successive equivalences 2}
        \dCat(X_{k-1} \times C_k \times C_{k+1} \times \dots \times C_n)^{G_k^+} \simeq \dCat(X_k \times C_{k+1} \times \dots \times C_n)
    \end{equation}
    for any $1 \leq k \leq n$.
    Note by Lemma~\ref{successive invariants} we have
    \begin{equation}\label{eq:successive invariants 2}
            \dCat(C_1 \times \dots \times C_n) ^{G_n} \simeq \left( \dCat(C_1 \times \dots \times C_n) ^{G_{n-1}} \right)^{G_n^+},
    \end{equation}
    so we can apply \eqref{eq:successive invariants 2} and \eqref{eq:successive equivalences 2} recursively to conclude $\dCat(C_1 \times \dots \times C_n)^{G_n} \simeq \dCat(X_n)$.

    Now we prove the claim by induction. 

    \renewcommand{\theIH}{(IH)}  
    \begin{IH}\label{IH2} The following conditions hold for all $n \geq 1$.
        \begin{enumerate}
            \item The morphism $\tau_n\colon X_{n} \to \overline X_{n}$ is a crepant resolution and
            $\dCat(X_{n-1} \times C_n)^{G_n^+} \simeq \dCat(X_n)$. 
            \item The fixed locus $X_n^{G_n^-}$ is a divisor. 
        \end{enumerate}       
    \end{IH}
    We note that the second hypothesis implies that for any $Z_n$ in $X_n^{G_n^-}$, the induced $ (\id_{X_{n-1}} , -\id )^{[2]}$-action on the tangent space is 
            \begin{equation*}
                \diag ( \overbrace{1,\dots,1}^{n-1}, -1).
            \end{equation*}
            
    It is clear for $n=1$.
    Assume we already know \ref{IH2} is true for $n=k$.
    Since $X_k^{G_k^-}$ is a divisor.
    the fixed locus $(X_{k} \times C_{k+1})^{G_{k+1}^+}$ by the even action $G_{k+1}^+$ is of codimension $2$,
    and the induced action on the tangent space $T_z$ of any fixed point $z$ is
    \begin{equation}\label{eq:action on tangent space}
        \diag(1, \dots ,1,-1,-1) \in \SL({k+1}, \CC)
    \end{equation}
    which implies that $\omega_{X_{k} \times C_{k+1}}$ is locally trivial as a $G_{k+1}^+$-sheaf.
    
    Now consider the nontrivial fiber of $\tau_{k+1}$.
    For any fixed point $z$ in $(X_{k} \times C_{k+1})^{G_{k+1}^+}$, the fiber $\tau_{k+1}^{-1}(\bar z)$ consists of the $G_{k+1}^+$-clusters supported on $z$, where $\bar z$ denotes the image of $z$ in the quotient $(X_k \times C_{k+1})/G_{k+1}^+$.
    Note that a double point $Z_{k+1}$ supported on $z$ corresponds to a direction $[w_1: \dots : w_{k+1}] \in \PP^{k}$ in the tangent space $T_z$. By \eqref{eq:action on tangent space}, the $G_{k+1}^+$-invariance forces
    \begin{equation*}
        [w_1 : \dots : w_{k+1}] = [w_1 : \dots : w_{k-1} : - w_{k} : - w_{k+1}].
    \end{equation*}
    Therefore, either $w_{k} = w_{k+1} =0$ or $w_1= \dots =w_{k-1}=0$.
    For the first case, the action on $Z_{k+1}$ is trivial, hence $\cO_{Z_{k+1}}$ cannot be isomorphic to the regular representation.
    So it has to be the second case, where the fiber $\tau_{k+1}^{-1}(\bar z)$ 
    \begin{equation}\label{eq:fiber}
        [0: \dots :0: w_k :w_{k+1}]
    \end{equation}
    is $1$-dimensional,
    hence $\dim \left( X_{k+1} \times_{\overline{X}_{k+1}} X_{k+1} \right) \leq k-1 + 2 \times 1 \leq k+2$.
    We have checked all the conditions in Theorem~\ref{BKR}, and by this theorem we know that the first part of \ref{IH2} is true for $n=k+1$.    

    Now consider $Z_{k+1}$ that is $G_{k+1}^-$-invariant.
    By definition, the subscheme $Z_{k+1} \subset X_{k} \times C_{k+1}$ is invariant under both $(\id_{X_{k}}, -\id)^{[2]}$ and $\left( {(\id_{X_{k-1}}, -\id )}^{[2]} ,-\id \right)^{[2]}$.
    So $Z_{k+1}$ is either a free orbit or a double point supported on $X_{k}^{G_k^-} \times C_{k+1}^{\langle -\id \rangle}$.
    \begin{itemize}
        \item Suppose that the subscheme $Z_{k+1}$ is a free orbit.
        Then it must be contained in $X_{k} \times C_{k+1}^{ \langle -\id \rangle}$ or $ X_{k}^{ G_k^-} \times C_{k+1}$, otherwise the orbit will contain more than $2$ points.
        It is easy to see that they form a (non-closed) subset of $X_{k+1}$ of codimension $1$.
        \item Suppose that the subscheme $Z_{k+1}$ is a double point (set-theoretically) supported on some point $z$ on $X_{k}^{G_k^-} \times C_{k+1}^{\langle -\id \rangle}$.
        We may assume that the subscheme $Z_{k+1}$ corresponds to the direction \eqref{eq:fiber} as the discussion before, and by the extra $G_{k+1}^-$-invariance, we have $[w_k : w_{k+1}] = [w_k : - w_{k+1}]$. So it has to be 
        \begin{equation*}
            [0: \dots : 0 : 1 :0] \quad \text{or} \quad [0: \dots : 0 : 0 :1],
        \end{equation*}
        which are precisely the limits of free orbits discussed in the previous case.
    \end{itemize}
    Therefore, we have checked the second part of \ref{IH2} for $n=k+1$.
\end{proof}

Now we prove that there exist stability conditions on $X_n$.
\begin{proof}[Proof of Theorem~\ref{thm:CH}, the $\ZZ/2$ case]
    According to Proposition~\ref{special stability via curve product} and Theorem~\ref{pre-restriction to end}, we know there exist $G_n$-invariant stability conditions on $\dCat(C_1 \times \dots \times C_n)$. 
    By Theorem~\ref{inducing stability conditions}, they induce stability conditions on $\dCat(C_1 \times \dots \times C_n)^{G_n}$.
\end{proof}

\subsection{Stability conditions on Cynk--Hulek Calabi--Yau manifolds arising from $\ZZ/3$-actions} \label{section:CH Z/3}
Let $C_1,\dots,C_n$ be elliptic curves with an order $3$ automorphism $\zeta$, and
\begin{equation*}
    G_n \defeq \{ (a_1 , \dots, a_n) \mid \sum a_i=0 \} \subset (\ZZ/3)^n
\end{equation*}
act on $C_1 \times \dots \times C_n$ via factor-wise order $3$ automorphism.
A Cynk--Hulek Calabi--Yau manifold arising from $\ZZ/3$-actions is constructed as a crepant resolution of $(C_1 \times \dots \times C_n)/G_n$. Explicitly, for any $n\geq 2$ we define
\begin{equation*}
    X_n \defeq \Hilb_\can^{G^+_n}(X_{n-1} \times C_n)
\end{equation*}
where the $\ZZ/3 \simeq G_n^+$-action on $X_{n-1} \times C_n$ is described inductively in the following.
The natural action $(\id_{X_{n-2}}, \zeta)$ on $X_{n-2} \times C_{n-1}$ gives rise to an action on $X_{n-1}= \Hilb_\can^{G_{n-1}^+}(X_{n-2} \times C_{n-1})$ which we shall denote by $(\id_{X_{n-2}}, \zeta)^{[3]}$. 
Then the $\ZZ/3 \simeq G_n^+$-action on $X_{n-1} \times C_n$ is generated by $\left( (\id_{X_{n-2}}, \zeta)^{[3]},\zeta^2 \right)$.
To distinguish between these two actions, we will call the action generated by $(\id_{X_{n-2}}, \zeta)^{[3]}$ the \emph{odd action} and denote it by $G_{n-1}^-$, and the action generated by $\left( (\id_{X_{n-2}}, \zeta)^{[3]},\zeta^2 \right)$ the \emph{even action} and denote it by $G_n^+$.
We denote the quotient map by
\begin{equation*}
    \tau_n \colon X_n = \Hilb^{G_n^+}_\can(X_{n-1} \times C_n) \to (X_{n-1} \times C_n) / {G_n^+} \eqdef \overline{X}_n .
\end{equation*}
We adopt the convention
\begin{equation*}
    X_0= \emptyset, \quad X_1 = C_1, \quad \text{and} \quad  G_1^+=0
\end{equation*}
for $n=1$.

\begin{proposition}\label{equivalence Z/3}
    There is a Fourier--Mukai equivalence
    \begin{equation*}
        \dCat(C_1 \times \dots \times C_n) ^{G_n} \simeq \dCat(X_n)
    \end{equation*}
    for any $n\geq 1$.
\end{proposition}

Before the proof, we study $\ZZ/3$-clusters which are not free orbits. By abuse of notation, we use $\zeta$ to denote a cubic root of $1$ as well.
\begin{lemma}

    Let $o$ be a fixed point of some $\ZZ/3$-action with the induced action on the cotangent space being
    \begin{equation*}
        T^\vee_{o,1} \oplus T^\vee_{o,\zeta} \oplus T^\vee_{o,\zeta^2}.
    \end{equation*}
    Suppose $Z$ is a $\ZZ/3$-cluster supported on $o$.
    Then $Z$ is in one of the following forms:
    \begin{itemize}
        \item \emph{(Linear)} $Z= \spec \CC[x]/x^3$, where $x \in \left(T^\vee_{o,\zeta} \cup T^\vee_{o,\zeta^2} \right) \backslash \{ 0 \}$;
        \item \emph{(Planar)} $Z= \spec \CC[x_1,x_2]/(x_1,x_2)^2$, where $x_i \in T^\vee_{o,\zeta^i} \backslash \{0\}$;
        \item \emph{(Curvilinear)} $Z =\spec \CC[x_1,x_2] /(x_1^3, x_1^2 - u x_2) $ or $\spec \CC[x_1,x_2] /(x_2^3, x_2^2 - v x_1)$, where $u,v \in \CC^\times$ and $x_i \in T^\vee_{o,\zeta^i} \backslash \{0\}$.
    \end{itemize}
\end{lemma}
\begin{proof}
    It is easy to check all these subschemes listed above are indeed $\ZZ/3$-clusters, and it suffices to prove that there are no more.
    Consider $\cO_Z$. It contains $\overline 1$, so there is no other eigenvector of $1$.
    We may assume $\cO_Z$ contains some nonzero $\overline x$ such that $x \in T^\vee_{o,\zeta} \cup T^\vee_{o,\zeta^2}$. Without loss of generality, we may assume $x \in T^\vee_{o,\zeta}$.
    Note that it forces that $\cO_Z$ contains no nonzero $\overline x'$ in the same eigenspace.    

    Suppose that $\cO_Z$ contains no $\overline{y} \in T^\vee_{o,\zeta^2}$.
    Then we may assume $Z$ is a closed subscheme of $\spec\CC[x]$.
    Note that $\overline x^3$ has the same eigenvalue with $\overline 1$, so it forces $\overline x^3 =0$.
    Then $x^3$ already cuts a length $3$ subscheme of $\CC[x]$, so it forces
    $Z= \spec \CC[x]/x^3$, i.e. $Z$ is linear.

    Now suppose that $\cO_Z$ contains some nonzero $\overline{y} \in T^\vee_{o,\zeta^2}$, so we may assume $Z$ is a closed subscheme of $\spec \CC[x,y]$.
    Since $\overline{x}^3,\overline{y}^3$, and $\overline{x}\overline{y}$ has the same eigenvalue with $\overline 1$, we know $\overline{x}^3 = \overline{y}^3 =\overline{x}\overline{y} =0$.
    So $\overline x^2 - u \overline y =0 $ and $\overline y^2 - v \overline x=0$ for some $u,v \in \CC$. If $u \neq 0$, then $(x^3,x^2-u y)$ already cut a length $3$ subscheme, i.e. $\cO_Z$ is curvilinear.
    The same for $v \neq 0$.
    So it remains the case $u=0$ and $v=0$.
    Then $(x^2,xy,y^2)$ cut a length $3$ subscheme,
    and we are in the planar case.
\end{proof}


\begin{corollary}\label{Z/3-cluster}
    Let $o$ be a fixed point of some $\ZZ/3$-action and let $g$ be a generator.
    \begin{enumerate}
        \item Suppose the induced action of $g$ on the cotangent space $T^\vee_o$ is 
        \begin{equation*}
            \diag ( \overbrace{1,\dots,1}^{n-2}, \zeta , \zeta^2)
        \end{equation*}
        and let $x_1,x_2$ be a choice of eigenvectors, say $x_i \in T^\vee_{o,\zeta^i}$.
        Then $Z\in P \cup Q$, where 
        \begin{align*}
            P \defeq & \{ \spec \CC[x_1,x_2] /(u_0 x_1^2-u_1 x_2, x_1 x_2, x_2^2) \mid [u_0 :u_1] \in \PP^1 \} \\
            = & \{ \spec \CC[x_1,x_2] /(x_1^3, \frac{u_0}{u_1} x_1^2 - x_2) \mid \frac{u_0}{u_1} \in \CC \} \cup \{ \spec \CC[x_1,x_2] /(x_1, x_2)^2 \} 
        \end{align*}
        and
        \begin{align*}
            Q \defeq & \{ \spec \CC[x_1,x_2] /(x_1^2, x_1 x_2, v_0 x_2^2-v_1 x_1) \mid [v_0 :v_1] \in \PP^1 \} \\
            = & \{ \spec \CC[x_1,x_2] /(x_2^3, \frac{v_0}{v_1} x_2^2 - x_1) \mid \frac{v_0}{v_1} \in \CC \} \cup \{ \spec \CC[x_1,x_2] /(x_1, x_2)^2 \}.
        \end{align*}

        \item Suppose the induced action of $g$ on the cotangent space $T^\vee_o$ is
        \begin{equation*}
            \diag ( \overbrace{1,\dots,1}^{n-3}, \zeta^2 , \zeta^2 , \zeta^2 )
        \end{equation*}
        Then $Z \in R$, where
        \begin{equation*}
            R \defeq \{ \spec \CC[x]/x^3 \mid  x \in T^\vee_{o,\zeta^2}   \} \simeq \PP^2.
        \end{equation*}
    \end{enumerate}
\end{corollary}

\begin{proof}[Proof of Proposition~\ref{equivalence Z/3}]
    We claim that
    \begin{equation*}
        \text{$\tau_n$ is a crepant resolution} \quad \text{and} \quad \dCat(X_{n-1} \times C_n)^{G_n^+} \simeq \dCat(X_n) .
    \end{equation*}
    Once we know this, we have 
    \begin{equation}\label{eq:successive equivalences}
        \dCat(X_{k-1} \times C_k \times C_{k+1} \times \dots \times C_n)^{G_k^+} \simeq \dCat(X_k \times C_{k+1} \times \dots \times C_n)
    \end{equation}
    for any $1 \leq k \leq n$.
    Note by Lemma~\ref{successive invariants} we have
    \begin{equation}\label{eq:successive invariants}
            \dCat(C_1 \times \dots \times C_n) ^{G_n} \simeq \left( \dCat(C_1 \times \dots \times C_n) ^{G_{n-1}} \right)^{G_n^+},
    \end{equation}
    so we can apply \eqref{eq:successive invariants} and \eqref{eq:successive equivalences} recursively to conclude $\dCat(C_1 \times \dots \times C_n)^{G_n} \simeq \dCat(X_n)$.

    Now we prove the claim by induction. More precisely, we do induction on the following statement. Recall that
    \begin{equation*}
        G_{n}^-  \text{ acts on $X_n$ via} \langle (\id_{X_{n-1}} , \zeta )^{[3]}\rangle\quad \text{and} \quad G_{n}^+  \text{ acts on $X_{n-1} \times C_n$ via } \left\langle \left( (\id_{X_{n-2}}, \zeta)^{[3]},\zeta^2 \right) \right\rangle.
    \end{equation*}

    \renewcommand{\theIH}{(IH)}  
    \begin{IH}\label{IH} The following conditions hold for $n\geq 1$.
        \begin{enumerate}
            \item The morphism $\tau_n\colon X_{n} \to \overline X_{n}$ is a crepant resolution and
            $\dCat(X_{n-1} \times C_n)^{G_n^+} \simeq \dCat(X_n)$. 
            \item The fixed locus has a decomposition
            \begin{equation*}
                X_{n}^{G_n^-}= B_1(X_n) \sqcup B_2(X_n), 
            \end{equation*}
            where $B_1(X_n)$ is a divisor, such that for any $Z_n\in B_1(X_n)$, the induced $ (\id_{X_{n-1}} , \zeta )^{[3]}$-action on the tangent space is 
            \begin{equation*}
                \diag ( \overbrace{1,\dots,1}^{n-1}, \zeta)
            \end{equation*}
            and $B_2(X_n)$ is of codimension $2$,
            such that for any $Z_n \in B_2$, the induced $ (\id_{X_{n-1}} , \zeta )^{[3]}$-action on the tangent space is 
            \begin{equation*}
                \diag ( \overbrace{1,\dots,1}^{n-2}, \zeta^2 , \zeta^2).
            \end{equation*}
            For $n=1$, this is understood as $B_2(X_1)= \emptyset$.
        \end{enumerate}       
    \end{IH}
    \addtocounter{IH}{-1}  

    It is clear for $n=1$.
    Assume we already know \ref{IH} is true for $n=k$.
    By \ref{IH} for $n=k$, we know the fixed locus the odd action $G_k^-$ is $B_1(X_k) \sqcup B_2(X_k)$ and we know the induced action on the tangent space of any fixed point.
    Consequently, the fixed locus by the even action $G_{k+1}^+$ is 
    \begin{equation*}
        B_1(X_k) \times C_{k+1}^{\langle \zeta \rangle} \quad \sqcup \quad B_2(X_k) \times C_{k+1}^{\langle \zeta \rangle}
    \end{equation*}
    with the induced action on tangent space of any fixed point generated by
    \begin{equation*}
        \diag (\overbrace{1, \dots, 1}^{k-1}, \zeta, \zeta^2 ) \quad \text{or} \quad \diag  ( \overbrace{1, \dots , 1}^{k-2} , \zeta^2 , \zeta^2, \zeta^2)
    \end{equation*}
    respectively.
    Both of them lie in $\SL(k+1,\CC)$, hence $\omega_{X_{k} \times C_{k+1}}$ is locally trivial as a $G_{k+1}^+$-sheaf, and by Corollary~\ref{Z/3-cluster}, we know the fiber over the first type is $1$-dimensional and over the second type is $2$-dimensional, thus
    \begin{equation*}
        \dim \left( X_{k+1} \times_{\overline{X}_{k+1}} X_{k+1} \right) = \max \{k+1,k-1+2 \times 1, k-2+ 2 \times 2\}  = k+2.
    \end{equation*}
    Therefore, we can apply Theorem~\ref{BKR}, and we obtain that 
    \begin{equation*}
        \text{$\tau_{k+1}$ is a crepant resolution} \quad \text{and} \quad \dCat(X_{k} \times C_{k+1})^{G_{k+1}^+} \simeq \dCat(X_{k+1}).
    \end{equation*}
    So we proved the first part.
    
    Now consider $Z_{k+1} \in X_{k+1}$ that is invariant under the odd action $G_{k+1}^-=\langle (\id_{X_{k}}, \zeta)^{[3]} \rangle$.
    By definition, $Z_{k+1} \subset X_k \times C_{k+1}$ is invariant under both $(\id_{X_{k}}, \zeta)$ and $\left( (\id_{X_{k-1}},\zeta)^{[3]},\zeta^2 \right)$. So $Z_{k+1}$ is either a free orbit or a triple point supported on one point in $X_{k}^{G_k^-} \times C_{k+1}^{ \langle \zeta \rangle}$. 
    \begin{itemize}
        \item Suppose that the subscheme $Z_{k+1}$ is a free orbit.
        Then it must be contained in $X_{k} \times C_{k+1}^{ \langle \zeta \rangle}$ or $ X_{k}^{ G_k^-} \times C_{k+1}$, otherwise the orbit will contain more than $3$ points.
        By \ref{IH} for $n=k$, we write
        \begin{equation*}
            X_{k}^{ G_k^- } = B_1(X_{k})\cup B_2(X_{k}).
        \end{equation*}
        We know $Z_{k+1}$ is either an element of 
        \begin{equation*}
             B_1^\circ(X_{k+1}) \defeq \{ Z_{k+1} \text{ is a free $G_{k+1}^+$-orbit} \mid \text{$Z_{k+1} \subset X_{k} \times C_{k+1}^{ \langle \zeta \rangle}$ or $Z_{k+1} \subset B_1(X_k) \times C_{k+1}$} \}
        \end{equation*}
        or an element of
        \begin{equation*}
            B_2^\circ (X_{k+1}) \defeq \{ Z_{k+1} \text{ is a free $G_{k+1}^+$-orbit} \mid \text{$Z_{k+1} \subset B_2(X_k) \times C_{k+1}$} \}
        \end{equation*}
        For any $Z_{k+1} \in B_1^\circ(X_{k+1})$, it is easy to see that a generator of the odd action $G_{k+1}^-$ induces the action on $T_{Z_{k+1}}X_{k+1}$ as
        \begin{equation*}
            \diag (\overbrace{1, \dots, 1}^{k}, \zeta),
        \end{equation*}
        so we know 
        \begin{equation}\label{eq:B1 o}
            B_1^\circ (X_{k+1}) \subset B_1(X_{k+1}).
        \end{equation}
        
        Similarly, for any $Z_{k+1} \in B_2^\circ(X_{k+1})$, it is easy to see that a generator of the odd action $G_{k+1}^-$ induces the action on $T_{Z_{k+1}}X_{k+1}$ as
        \begin{equation*}
            \diag (\overbrace{1,\dots, 1}^{k-1}, \zeta^2,\zeta^2),
        \end{equation*}
        so we know 
        \begin{equation}\label{eq:B2 o}
            B_2^\circ (X_{k+1}) \subset B_2(X_{k+1}).
        \end{equation}

        \item Suppose that the subscheme $Z_{k+1}$ is a triple point (set-theoretically) supported on some point in $X_{k}^{G_k^-} \times C_{k+1}^{ \langle \zeta \rangle}$.
        Set $\{ z \} \defeq \supp (Z_{k+1})$.
        Again by \ref{IH} for $n=k$, we know
        \begin{equation*}
            z \in B_1(X_k) \times C_{k+1}^{\langle \zeta \rangle} \sqcup B_2(X_k) \times C_{k+1}^{\langle \zeta \rangle}
        \end{equation*}
        and for any such $z$, the induced $G_{k+1}^+$-action on the tangent space of any fixed point is generated by
        \begin{equation*}
            \diag (\overbrace{1, \dots, 1}^{k-1}, \zeta, \zeta^2 ) \quad \text{or} \quad \diag  ( \overbrace{1, \dots , 1}^{k-2} , \zeta^2 , \zeta^2, \zeta^2)
        \end{equation*}
        respectively.
        Note that the induced action of $(\id_{X_{k}}, \zeta)$ on the tangent space takes the form
        \begin{equation*}
            \diag( \overbrace{1 ,\dots , 1}^{k} , \zeta)
        \end{equation*}
        in the same coordinate system.
        \begin{itemize}
            \item Suppose $z \in B_1(X_k) \times C_{k+1}^{\langle \zeta \rangle}$, and let $x_i \in T^\vee_z(X_k \times C_{k+1})$ be an eigenvector of $\zeta^i$. Then by Corollary~\ref{Z/3-cluster}, we know
            $Z_{k+1}\in P_z \cup Q_z$, where 
            \begin{align*}
                P_z \defeq & \{ \spec \CC[x_1,x_2] /(u_0 x_1^2-u_1 x_2, x_1 x_2, x_2^2) \mid [u_0 :u_1] \in \PP^1 \} \\
                = & \{ \spec \CC[x_1,x_2] /(x_1^3, \frac{u_0}{u_1} x_1^2 - x_2) \mid \frac{u_0}{u_1} \in \CC \} \cup \{ \spec \CC[x_1,x_2] /(x_1, x_2)^2 \} 
            \end{align*}
            and
            \begin{align*}
                Q_z \defeq & \{ \spec \CC[x_1,x_2] /(x_1^2, x_1 x_2, v_0 x_2^2-v_1 x_1) \mid [v_0 :v_1] \in \PP^1 \} \\
                = & \{ \spec \CC[x_1,x_2] /(x_2^3, \frac{v_0}{v_1} x_2^2 - x_1) \mid \frac{v_0}{v_1} \in \CC \} \cup \{ \spec \CC[x_1,x_2] /(x_1, x_2)^2 \}.
            \end{align*}
            Among them, the $G_{k+1}^-$-invariant part, which is by definition just invariant under the assignment $x_2 \mapsto \zeta x_2$, is 
            \begin{equation*}
                (P_z \cup Q_z)^{G_{k+1}^-}= \{ P_z^\infty, Q_z^\infty ,P_z^0 =Q_z^0\},
            \end{equation*}
            where
            \begin{equation*}
                P_z^\infty \defeq \spec \CC[x_1,x_2]/(x_1^3,x_2) , \quad Q_z^\infty \defeq \spec \CC[x_1,x_2]/(x_2^3,x_1) ,
            \end{equation*}
            and
            \begin{equation*}
                P_z^0=Q_z^0 \defeq \spec \CC[x_1,x_2]/(x,y)^2 \in P_z \cap Q_z .
            \end{equation*}
            We claim the following statements which describes the tangent action.
            \begin{claim}
                The induced $(\id_{X_k} ,\zeta)^{[3]}$-action on $T_{P_z^\infty}(X_{k+1})$ is $\diag(1,\dots,1,\zeta)$.
            \end{claim}
            Indeed, the $(\id_{X_k} ,\zeta)^{[3]}$-action takes the form 
            \begin{equation*}
                x_2 \mapsto - \zeta x_2 = \zeta^2 x_2
            \end{equation*}
            on $T^\vee_{P_z^\infty}(X_{k+1})$ since it is contravariant. So the action on the ideals takes the form
            \begin{equation*}
                (x_1^3,\varepsilon x_1^2 -x_2) \mapsto (x_1^3,\varepsilon x_1^2 - \zeta^2 x_2) = (x_1^3,\varepsilon \zeta x_1^2 - x_2) ,
            \end{equation*}
            as claimed.

            \begin{claim}
                The induced $(\id_{X_k} ,\zeta)^{[3]}$-action on $T_{Q_z^\infty}(X_{k+1})$ is $\diag(1,\dots,1,\zeta)$.
            \end{claim}
            This is similar, just note that
            \begin{equation*}
                (x_2^3,\varepsilon x_2^2 -x_1) \mapsto (x_2^3,\varepsilon \zeta x_2^2 - x_1).
            \end{equation*}
            Therefore, we know
            \begin{equation}\label{eq:B1'}
                 B_1'(X_{k+1}) \defeq \cup_{z \in B_1(X_k) \times C_{k+1}^{\langle \zeta \rangle} } \{ P_z^\infty , Q_z^\infty \} \subset B_1(X_{k+1}) .
            \end{equation}

            Similarly, we have
            \begin{claim}
                The induced $(\id_{X_k} ,\zeta)^{[3]}$-action on $T_{P_z^0}(X_{k+1})$ is $\diag(1,\dots,1,\zeta^2,\zeta^2)$.
            \end{claim}
            For this, just note that
            \begin{equation*}
                (x_1^3, x_1^2 - \varepsilon x_2) \mapsto (x_1^3, x_1^2 -  \varepsilon \zeta^2 x_2)
            \end{equation*}
            and
            \begin{equation*}
                (x_2^3, x_2^2 - \varepsilon x_1) \mapsto (x_2^3, \zeta x_2^2 -  \varepsilon  x_1) = (x_2^3, x_2^2 -  \varepsilon \zeta^2 x_1).
            \end{equation*}
            So we know
            \begin{equation}\label{eq:B2 ess}
                 B_2^\dagger(X_{k+1}) \defeq \cup_{z \in B_1(X_k) \times C_{k+1}^{\langle \zeta \rangle} } \{ P_z^0 \} \subset B_2(X_{k+1}) .
            \end{equation}

            \item Suppose $z \in B_2(X_k) \times C_{k+1}^{\langle \zeta \rangle}$, and let $x_1,x_2,x_3$ be eigenvectors of the last $3$ components.
            Again by Corollary~\ref{Z/3-cluster}, we know $Z_{k+1} \in R_z$, where 
            \begin{equation*}
                R_z \defeq \{ \spec \CC[x]/x^3 \mid x = w_1 x_1 + w_2 x_2 + w_3 x_3  , [w_1 : w_2 :w_3] \in \PP^2 \}.
            \end{equation*}
            And among them, the $G_{k+1}^-$-invariant part is
            \begin{equation*}
                R_z^{G_{k+1}^-} =\{ R_z^{0,w} \}_{w \in \PP^1} \cup \{R_z^\infty\}
            \end{equation*}
            where
            \begin{equation*}
                R_z^{0,w} \defeq \spec \CC[x]/x^3 \quad \text{with} \quad x= w_1 x_1 + w_2 x_2 , \quad [w_1 :w_2] \in \PP^1,
            \end{equation*}
            and
            \begin{equation*}
                R_z^\infty \defeq \spec \CC[x_3]/x_3^3.
            \end{equation*} 
            
            We claim that
            \begin{claim}
                The induced $(\id_{X_k} ,\zeta)^{[3]}$-action on $T_{R_z^{0,w}}(X_{k+1})$ is $\diag(1,\dots,1,\zeta)$.
            \end{claim}
            Indeed, the $(\id_{X_k} ,\zeta)^{[3]}$-action on the tangent space takes the form
            \begin{equation*}
                [w_1 : w_2 : \varepsilon] \mapsto [w_1 : w_2 : \varepsilon \zeta].
            \end{equation*}
            
            So we have
            \begin{equation} \label{eq:B1''}
                B_1'' (X_{k+1}) \defeq \cup_{z \in B_2(X_k) \times C_{k+1}^{\langle \zeta \rangle} } \{ R_z^{0,w} \} \subset B_1(X_{k+1}) .
            \end{equation}
            Similarly, we claim that
            \begin{claim}
                The induced $(\id_{X_k} ,\zeta)^{[3]}$-action on $T_{R_z^\infty}(X_{k+1})$ is $\diag(1,\dots,1,\zeta^2 ,\zeta^2)$.
            \end{claim}
            Indeed, the $(\id_{X_k} ,\zeta)^{[3]}$-action on the tangent space takes the form
            \begin{equation*}
                [ \varepsilon_1 : \varepsilon_2 : 1] \mapsto [ \varepsilon_1 : \varepsilon_2 : \zeta] = [ \varepsilon_1 \zeta^2 : \varepsilon_2 \zeta^2 : 1]
            \end{equation*}
            
            So we have
            \begin{equation} \label{eq:B2'}
                B_2' (X_{k+1}) \defeq \cup_{z \in B_2(X_k) \times C_{k+1}^{\langle \zeta \rangle}}\{  R_z^\infty \} \subset B_2(X_{k+1}).
            \end{equation}
        \end{itemize}
    \end{itemize}
    So far we have discussed all possible situations.
    By \eqref{eq:B1 o}, \eqref{eq:B1'}, and \eqref{eq:B1''}, we know
    \begin{equation*}
        B_1(X_{k+1})=B_1^\circ(X_{k+1}) \cup B_1'(X_{k+1}) \cup B_1''(X_{k+1}) ,
    \end{equation*}
    and by \eqref{eq:B2 o}, \eqref{eq:B2'}, and \eqref{eq:B2 ess}, we know
    \begin{equation*}
        B_2(X_{k+1})=B_2^\circ(X_{k+1}) \cup B_2'(X_{k+1}) \sqcup B_2^\dagger(X_{k+1}) .
    \end{equation*}
    Finally, note that $\dim(B_1^\circ(X_{k+1}))=k$ and 
    \begin{equation*}
        \overline{B_1^\circ(X_{k+1}) } \supset B_1'(X_{k+1}) \cup B_1''(X_{k+1}),
    \end{equation*}
    we know $B_1(X_{k+1})$ is indeed a divisor.
    Similarly, since
    $\dim(B_2^\circ(X_{k+1}))=k-1$
    and
    \begin{equation*}
        \overline{B_2^\circ(X_{k+1}) } \supset B_2'(X_{k+1}),
    \end{equation*}
    we know $B_2(X_{k+1})$ is of codimension $2$.
    Therefore, we have proved \ref{IH} for n=k+1.
\end{proof}

\begin{remark}
    In the paper \cite{Cynk-Hulek:2007} they also construct Calabi--Yau manifolds from elliptic curves with $\ZZ/4$-actions. 
    However, the argument does not apply to this case since the dimension condition of Theorem~\ref{BKR} does not hold.
\end{remark}

\begin{proof}[Proof of Theorem~\ref{thm:CH}, the $\ZZ/3$ case]
    Again, by Proposition~\ref{special stability via curve product} and Theorem~\ref{pre-restriction to end}, we know there exist $G_n$-invariant stability conditions on $\dCat(C_1 \times \dots \times C_n)$. 
    By Theorem~\ref{inducing stability conditions}, they induce stability conditions on $\dCat(C_1 \times \dots \times C_n)^{G_n}$.
\end{proof}

\newpage

\appendix

\bibliographystyle{alpha}        
\bibliography{main}

\bigskip 
{Email address: \texttt{y.cheng@imperial.ac.uk}}

\end{document}